\documentclass[12pt,reqno]{amsart}
\usepackage{graphicx} 
\usepackage{fullpage}

\usepackage{etoolbox}

\usepackage[utf8]{inputenc}
\usepackage{amsmath}
\usepackage{amsfonts}
\usepackage{amsthm}
\usepackage{multicol}
\usepackage{pictexwd, dcpic}
\usepackage{fancyhdr}
\usepackage{mathrsfs}
\usepackage{latexsym,mathtools}
\usepackage{bm, enumerate}
\usepackage[shortlabels]{enumitem}
\usepackage{amssymb}
\usepackage[all,cmtip]{xy}
\usepackage[x11names]{xcolor}
\usepackage[linkcolor=Blue2,citecolor=red,colorlinks]{hyperref}
\usepackage[capitalize,nameinlink]{cleveref}

\usepackage{comment}

\newtheorem{thm}{Theorem}[section]

\newtheorem{lemma}[thm]{Lemma}
\newtheorem{prop}[thm]{Proposition}

\newtheorem{defn}[thm]{Definition}

\newtheorem{theorem}{Theorem}

\renewcommand{\mod}[1]{\left|#1\right|}

\newcommand{\bz}{\boldsymbol{z}}
\newcommand{\bw}{\boldsymbol{w}}
\newcommand{\bv}{\boldsymbol{v}}
\newcommand{\by}{\boldsymbol{y}}
\newcommand{\bx}{\boldsymbol{x}}

\newcommand{\bt}{\boldsymbol{t}}

\newcounter{egcounter}
\newenvironment{example}{\addtocounter{egcounter}{1}\noindent\textbf{Example \arabic{egcounter} }}{}

\newcommand{\cH}{{\mathcal H}}

\newcommand*{\mf}{\mathbf}

\newcommand\blfootnote[1]{%
  \begingroup
  \renewcommand\thefootnote{}%
  \footnote{#1}%
  \addtocounter{footnote}{-1}%
  \endgroup
}

\begin{document}
	\title[Multiplicative linear functionals on RKHS]{Multiplicative Linear Functionals on Reproducing Kernel Hilbert Spaces}

\author[T. Bhattacharyya]{Tirthankar Bhattacharyya$^1$}\blfootnote{$^1$Department of Mathematics, Indian Institute of Science, Bangalore 560012, India. Email: tirtha@iisc.ac.in}

\author[Jaikishan]{Jaikishan$^2$}\blfootnote{$^2$Department Of Mathematics, SNIoE (deemed to be university), School of Natural Sciences, Gautam Budh Nagar, Uttar Pradesh, 203207, India. Email: jk301@snu.edu.in}

\author[P. Kumar]{Poornendu Kumar$^3$}\blfootnote{$^3$Faculty of Mathematics and Physics, University of Ljubljana, Jadranska ulica 19, 1000 Ljubljana, Slovenia. Email: poornendu.kumar@fmf.uni-lj.si, poornendukumar@gmail.com}

\author[C. K. Sahu]{Chaman Kumar Sahu$^4$}\blfootnote{$^4$Department of Mathematics, Indian Institute of Technology Bombay, Powai, Mumbai, 400076, India. Email: sahuchaman9@gmail.com, chamanks@math.iitb.ac.in}

	\thanks{2020 {\em MSC}: 46E22, 47B32, 46H99, 30B50.\\
	{\em Keywords and phrases}: Complete Pick kernel, Composition operator, Diagonal Kernel, Dirichlet series, Multiplicative linear functional, Reproducing kernel Hilbert space.
}

\maketitle

\begin{abstract}
The classical Gleason–Kahane–\.Zelazko theorem characterizes multiplicative linear functionals on a unital Banach algebra through the scalar identity $\Lambda(x^{2})=\Lambda(x)^{2}$. We develop analogues of this theorem for bounded linear functionals on reproducing kernel Hilbert spaces of holomorphic functions on domains in $\mathbb{C}^{d}$, replacing conditions on the whole space by tractable conditions involving only kernel functions.

Our first main result shows that if $k$ is a diagonal holomorphic kernel on a domain $\Omega\subseteq\mathbb{C}^{d}$ containing the origin, and if $k_{\bw}^{2}\in\mathcal H(k)$ for every $\bw\in\Omega$, then a bounded linear functional $\Lambda$ on $\mathcal H(k)$ satisfying $\Lambda(1)=1$ is multiplicative if and only if $\Lambda(k_{\bw}^{2})=\Lambda(k_{\bw})^{2}$ for all $\bw\in\Omega$. Kernels satisfying $2$-point Pick property and their powers furnish a natural class of examples.

When $k$ arises as a Schur product or a tensor product of complete diagonal Pick kernels, we obtain a further, more algebraic characterization of multiplicativity, expressed in terms of the values of $\Lambda$ on kernel functions and their reciprocals. This framework subsumes the weighted Bergman and Dirichlet-type spaces, as well as the Hardy space on the polydisc.

We extend the analysis to Hilbert spaces associated with diagonal Dirichlet series kernels on half-planes, encompassing in particular the Hardy space of Dirichlet series and its Riemann zeta reproducing kernel. Explicit examples demonstrate that the boundedness hypothesis on $\Lambda$ cannot be omitted. 

Finally, our characterization of multiplicative linear functionals leads to characterizations of weighted composition operators on a reproducing kernel Hilbert space associated with a diagonal holomorphic kernel. 
\end{abstract}

{\footnotesize \noindent \textsf{Funding}: This work is supported by J C Bose Fellowship number JCB/2021/000041 of ANRF, the  DST FIST program 2021 (TPN-700661),  the Slovenian Research Agency (program P1-0222 and grant J1-50002) and the COMPUTE project under the QuantERA II Programme, funded by the European Union’s Horizon 2020 programme (Grant Agreement No. 101017733) and the Institute Postdoctoral Fellowship of the Indian Institute of Technology Bombay.}

\newpage

\section{When is a bounded linear functional multiplicative?}
\subsection{Motivation:} 
Multiplicative linear functionals have long played an important role in the study of commutative Banach algebras. In an early contribution, Hewitt and Kakutani investigated in \cite{Hewitt-Kakutani-old, Hewitt-Kakutani}, multiplicative linear functionals on the measure algebra of a locally compact abelian group, constructing nonclassical examples and demonstrating the richness of its maximal ideal space. In particular, their main construction yields multiplicative linear functionals lying in the closure of the character group.

A fundamental general characterization of multiplicative linear functionals on Banach algebras is provided by the celebrated Gleason--Kahane--\.Zelazko theorem.

\noindent	\textbf{Theorem GK\.Z.}  (Gleason, Kahane and \.Zelazko, see \cite{Gleason}, \cite{KZ} and \cite{Zelazko}.) 
{\em Let $\mathcal A$ be a complex unital Banach algebra with identity $e$, and let $\Lambda:\mathcal A\to\mathbb C$ be a linear functional satisfying $\Lambda(e)=1$. Then TFAE: 
\begin{enumerate}
     \item[(a)]$\Lambda$ is multiplicative;
    \item [(b)]$\Lambda(x^2)=\Lambda(x)^2$ for all $x\in\mathcal A$; 
    \item [(c)]$\Lambda(x)\neq 0$ for all invertible element $x\in\mathcal A$. 
\end{enumerate}}

Numerous variants and extensions of this theorem have been established in different settings; see, for example \cite{CHMR, jaikishan2024multiplicativity, mashreghi2015gleason, mashreghiDirichlet,  roitman1981linear, sampat2021cyclicity}. We shall be concerned with (bounded) linear functionals on Hilbert function spaces.
\begin{defn}
  A linear functional $\Lambda$ on a reproducing kernel Hilbert space (RKHS) $\mathcal H(k)$ on a domain $\Omega$ is called multiplicative if $\Lambda(fg) = \Lambda(f)\Lambda(g)$ whenever $f$ and $g$ are such functions in $\mathcal H(k)$ that the pointwise product $fg$ is also in $\mathcal H(k)$.
\end{defn}

This is related to algebraic consistency, see \cite{CM} and \cite{Hartz-Can}. Results of GK\.Z type on Hardy spaces of the unit disc was harvested in \cite{mashreghi2015gleason} by the ingenious methodology of proving a GK\.Z theorem for a left module over a Banach algebra and then viewing the Hardy space as a left module over its multiplier algebra. 

{\em The multiplier algebra $M(\mathcal X)$ of a Banach space $\mathcal X$ consists of those holomorphic functions $\varphi$ on $\Omega$ which have the property that $f \mapsto \varphi f$ is a bounded operator on $\mathcal X$ with the norm
$$ \| \varphi \|_{M(\mathcal X)} = \sup \{ \|  \varphi f \|_{\mathcal X} : \| f \|_{\mathcal X} \le 1\}.$$ }
We shall be concerned with a Hilbert function space $\mathcal H$. $M(\mathcal H)$ is a closed subalgebra of the algebra of all bounded linear operators on $\mathcal H$, so it is a Banach algebra. For each $\bw \in \Omega$,
the evaluation functional $\varphi \rightarrow \varphi(\bw)$ is a character on $M(\mathcal H)$, so $|\varphi(\bw)| \le \|\varphi \|_{M(\mathcal H)}$. It follows
that $M(\mathcal H) \subseteq  H^\infty(\Omega)$. If $\mathcal H$ contains the constants,  $M(\mathcal H) \subseteq  \mathcal H$. 

The landmark module theorem was used again in \cite{CHMR}. This time the authors proved a GK\.Z theorem for complete Pick spaces. It is important to identify a special class of functions in the Hilbert space (like the invertible elements in a Banach algebra). The non-vanishing property of the linear functional on this class guarantees multiplicativity. In the instance of \cite{mashreghi2015gleason}, the special class consisted of outer functions. In \cite{CHMR}, it was the class of cyclic functions. The challenge with this result is that in many spaces there is no characterization of cyclic functions. Our motivating example is the Dirichlet space where this is a long-standing open problem, see \cite{BS}. Thus one looks for an alternative class of functions.  We do need more than just the non-vanishing property of the {\em bounded} linear functional though. Indeed, suppose for an element $\boldsymbol w$ of $\Omega $,  the notation $k_{\boldsymbol w}$ denotes the {\em kernel function} whose value at $\boldsymbol{z} $ is $k(\boldsymbol{z},\boldsymbol{w})$. Let $\Lambda$ be the bounded linear functional on the Hardy space $H^2(\mathbb D)$ given by $$\Lambda(f)=f(0)+f'(0).$$ Since $\Lambda(k_w)=1+\overline{w}\neq0$ for every $w\in\mathbb D$, $\Lambda$ does not vanish on any reproducing kernel. However, $\Lambda$ is not multiplicative.

 Kernel functions are cyclic in the Hardy space, and more generally in several important
 reproducing kernel Hilbert spaces, including complete Pick spaces; see, for example, \cite{CHMR}.  The preceding example shows, however, that nonvanishing on kernel functions alone is too weak to characterize multiplicativity. This suggests that, in seeking an analogue of the GK\.Z theorem for reproducing kernel Hilbert spaces, one should replace the family of all elements of a Banach algebra by the much smaller family of reproducing kernels, while retaining stronger algebraic identities. Accordingly, we investigate kernel-function analogues of two equivalent formulations of the
 GK\.Z theorem. The first is the square-preserving condition
 $
 \Lambda(k_{\bw}^2)=\Lambda(k_{\bw})^2,$ for $\bw\in\Omega,$
 which may be viewed as a kernel-level counterpart of $
 \Lambda(x^2)=\Lambda(x)^2.$  For the nonvanishing formulation in part (c), it is useful to observe that, for a normalized multiplicative functional on a unital Banach algebra, $
 \Lambda(x^{-1})=\frac{1}{\Lambda(x)}$ 
 for every invertible element $x$. This motivates the reciprocal kernel condition
 \[
 \Lambda\!\left(\frac{1}{k_{\bw}}\right)
 =
 \frac{1}{\Lambda(k_{\bw})}, \quad \bw\in\Omega
 \] 
 whenever $1/k_{\bw}$ belongs to the space. Our aim is to determine classes of reproducing kernel Hilbert spaces for which these
 identities, imposed only on reproducing kernels, already force \(\Lambda\) to be
 multiplicative. Thus, rather than testing \(\Lambda\) on all elements or all invertible
 elements, as in the classical GK\.Z theorem, we seek to detect
 multiplicativity from the distinguished family of kernel functions alone.

\subsection{Main Results} 
Let $\mathbb{Z}_+$ denote the set of all non-negative integers. For
$\alpha=(\alpha_1,\ldots,\alpha_d)\in\mathbb{Z}_+^d$ and
$\bz=(z_1,\ldots,z_d)\in\mathbb{C}^d$, we use the standard multi-index notation
$$
|\alpha|=\alpha_1+\cdots+\alpha_d,\qquad
\alpha!=\alpha_1!\cdots\alpha_d!,\qquad
\bz^\alpha=z_1^{\alpha_1}\cdots z_d^{\alpha_d}.
$$
For $\alpha\in\mathbb{Z}_+^d$, let
$$
\binom{|\alpha|}{\alpha}
=\frac{|\alpha|!}{\alpha_1!\cdots\alpha_d!},
 \quad \text{and define} \quad 
a_\alpha=
\begin{cases}
	\dfrac{a_{|\alpha|}}{\binom{|\alpha|}{\alpha}}, & \alpha\in\mathbb{Z}_+^d,\\[1ex]
	0, & \text{otherwise}.
\end{cases}
$$
For $\alpha,\gamma\in\mathbb Z_+^d$, we write $\gamma\leq\alpha$ if
$\gamma_j\leq\alpha_j$ for every $1\leq j\leq d$. Given $\bz,\bw\in\mathbb{C}^d$, we write
$$
\overline{\bw}=(\overline{w}_1, \ldots, \overline{w}_d) \quad \text{and} \quad \bz\odot\bw=(z_1w_1,\ldots,z_dw_d)
$$
for their coordinatewise product.  Let $\Omega\subseteq\mathbb C^d$ be a domain containing the origin. A reproducing kernel $k$ on $\Omega$ is called a \emph{diagonal holomorphic kernel} if it is holomorphic in the first variable, anti-holomorphic in the second variable, and there exists a family of positive numbers $\{a_\alpha\}_{\alpha\in\mathbb Z_+^d}$ such that the power series expansion of $k$ at $\textbf{0}$ is given by
	$$
	k(\bz,\bw)
	=
	\sum_{\alpha\in\mathbb Z_+^d}
	a_\alpha
	(\bz\odot\overline{\bw})^\alpha,
	\qquad
	\bz,\bw\in\Omega.
	$$
We call a diagonal holomorphic kernel \emph{normalized} if $a_{\mathbf 0}=1$. All unitarily invariant kernels on the Euclidean unit ball $\mathbb{B}_d$, as well as the classical Hardy, weighted Bergman, and Dirichlet-type kernels on the polydisc, are prime examples of this class. For a detailed discussion of these kernels, we refer the reader to \cite[Section~2.3]{MT}.
\begin{defn}
	A reproducing kernel Hilbert space $\mathcal H(k)$ is said to be irreducible if the reproducing kernel $k(\bz, \bw)$ is nonzero for all $\bz, \bw$ in $\Omega$, and the kernel functions $k_{\bz}$ and $k_{\bw}$ are linearly independent whenever $\bz \neq \bw$.
\end{defn}
Our first result characterizes the bounded multiplicative linear functionals on $\cH(k)$ associated with a normalized diagonal holomorphic kernel $k$. The boundedness assumption is essential, as demonstrated in \cref{examples:MainTheorems}.
\begin{theorem}\label{theorem_Main_1}
	Let $\Omega\subseteq\mathbb C^d$ be a domain containing $\mathbf 0$, and let
	$k$ be a diagonal holomorphic kernel on $\Omega$. Assume that $k_{\bw}^2$, the function whose value at a point $\bz$ is $k(\bz , \bw)^2$ is in $\mathcal H(k)$ for every $\bw\in\Omega.$ Suppose that $\Lambda:\mathcal H(k)\to\mathbb C$ is a bounded linear functional satisfying $\Lambda(1)=1$. Then the following statements are
	equivalent:
	\begin{enumerate}
		\item
		$\Lambda$ is multiplicative;
		\item
		$
		\Lambda(k_{\bw}^2)
		=
		\Lambda(k_{\bw})^2
		$ for every $\bw\in\Omega$.
	\end{enumerate}
\end{theorem}

A reproducing kernel $k$ on a set $\Omega$ is said to satisfy the {\em $N$-point Pick property} if, for every choice of distinct points
$\bz_1,\ldots,\bz_N\in\Omega$ and every choice of target values
$w_1,\ldots,w_N\in\mathbb D$, there exists a multiplier
$\varphi\in M(\mathcal H(k))$ with $
\|\varphi\|_{M(\mathcal H(k))}\le 1,$
such that $\varphi(\bz_j)=w_j,$ for $j=1,\ldots,N,$ if and only if the Pick matrix
\[
\left[\bigl(1-w_i\overline{w_j}\bigr)k(\bz_i,\bz_j)\right]_{i,j=1}^{N}\succeq 0.
\]
For a positive integer $p$, we shall denote by $k_{\bw}^{p}$ the function whose value at $\bz$ is $\bigl(k(\bz,\bw)\bigr)^p$. Note that if $k^{\circ p}$ denote the $p$-fold Schur power of $k$, then $k_{\bw}^{\circ p}(\bz) = k_{\bw}^{p}(\bz)$ as functions. Suppose that $k$ is normalized and satisfies the $2$-point Pick property. By Proposition~3.1 of \cite{Hartz-Can}, for every point $\bw$ in the domain there exists a strictly contractive multiplier $\psi_{\bw}$ such that $
\psi_{\bw}
=
1-\frac{1}{k_{\bw}}.$ Consequently,
\[
k_{\bw}
=
\frac{1}{1-\psi_{\bw}}
=
\sum_{n=0}^{\infty}\psi_{\bw}^{\,n},
\]
where the series converges absolutely in the multiplier algebra. It follows that $k_{\bw}\in M(\mathcal H(k)),$ and hence the hypothesis  of \cref{theorem_Main_1} is satisfied for such a $k$ and its powers. 

\begin{defn}\label{Pick}
		A kernel $k$ on a set $\Omega$ is called a complete Pick kernel if, for any numbers $N$ and $m$, any points $\bz_1, \bz_2, \ldots , \bz_N$ in $\Omega$ and any contractive $m \times m$ matrices $W_1, W_2,  \ldots , W_N$, the positivity of the matrix valued kernel $[(I - W_i^*W_j) k (\bz_i , \bz_j)]_{i,j=1}^N$ implies the existence of a contractive multiplier $\Phi$ on $\mathcal H(k) \otimes \mathbb C^m$ such that $\Phi(\bz_i) = W_i$ for all $i=1, \ldots, N$. 
\end{defn}

	Recall that complete Pick kernels (and the associated complete Pick spaces) arose from interpolation problems; see~\cite{AM_JFA}. Prototype examples of complete Pick spaces include the Hardy space $H^2(\mathbb{D})$ over the unit disc, the Drury--Arveson space over the unit ball and the Dirichlet space on the unit disc. The theory of complete Pick kernels is well developed, and we do not review it here.  

When the kernels possess additional structure, one obtains a further characterization that is more algebraic in nature. If $s_1,\ldots,s_p$ are reproducing kernels on $\Omega$, then
$s_1\circ\cdots\circ s_p$ denotes their Schur product, while
$s_1\otimes\cdots\otimes s_p$ denotes their tensor product kernel on
$\Omega\times\cdots\times\Omega$; see
\cite{Paulesn-Raghupathi} for details. If we write $k =s_1\otimes\cdots\otimes s_p$, then for $\bw = (\bw_1, \ldots, \bw_p) \in \Omega \times \cdots \times \Omega$, $k_{\bw}(\bz_1, \ldots ,\bz_p) = s_{1_{\bw_1}}(\bz_1) \cdots s_{p_{\bw_p}}(\bz_p)$ for all $(\bz_1, \ldots, \bz_p) \in \Omega \times \cdots \times  \Omega$. Our second main result is the following algebraic characterization.

\begin{theorem}\label{completePick}
	Let $\Omega$ be a domain in $\mathbb C^d$ containing $\mathbf 0$. Let $s_1,\ldots,s_p$ be  normalized, irreducible, diagonal complete  Pick kernels on $\Omega$. In the following, $\Lambda$ is  a bounded linear functional on $\mathcal H(k)$  satisfying $\Lambda(1)=1$.
		\begin{enumerate}
		\item If $k=s_1\circ\cdots\circ s_p$, then $s_{j_{\bw}}^{-1}\in M(\mathcal H(k)),$ for $j=1,\ldots,p,\ \bw\in\Omega$ and $\Lambda$ is multiplicative if and only if
		\[
		\frac{1}{\Lambda(k_{\bw})}
		=
		\prod_{j=1}^{p}
		\Lambda\!\left(\frac{1}{s_{j_{\bw}}}\right),
		\qquad
		\bw\in\Omega.
		\]
		
		\item If $k=s_1\otimes\cdots\otimes s_p$, then $s_{j_{\bw}}^{-1}\in M(\mathcal H(k)),$ for $j=1,\ldots,p,\ \bw\in\Omega$ and  $\Lambda$  is multiplicative if and only if
		\[
		\frac{1}{
			\Lambda(
			k_{\bw})}=\prod_{j=1}^{p}
			\Lambda\!\Big(\frac{1}{s_{j_{\bw_j}}}\Big), 
		\qquad
		\bw_1,\ldots,\bw_p\in\Omega.
		\]
	\end{enumerate}
\end{theorem}
For brevity, whenever we refer to a complete Pick kernel, we will always mean a normalized, irreducible, diagonal complete Pick kernel on $\Omega$.
Part~\textup{(1)} of \cref{completePick} applies in particular to kernels that are powers of normalized, irreducible complete Pick kernels; see \cref{Lemma_SP}. This class includes the weighted Dirichlet and weighted Bergman spaces. On the other hand, when each $s_j$ is the Szeg\"o kernel on $\mathbb D$, their tensor product is precisely the reproducing kernel of the Hardy space on the polydisc. Thus, the Hardy space on the polydisc provides a fundamental example covered by part~\textup{(2)}. The boundedness assumption is again essential, as illustrated in Example~4 of ~\cref{examples:MainTheorems}.

We now extend the scope of \cref{theorem_Main_1} by considering spaces of Dirichlet series. The study of the Hardy space of Dirichlet
series was initiated by Hedenmalm, Lindqvist and Seip in their seminal work \cite{HLS}, in which they characterized its multiplier algebra and resolved a completeness problem originally posed by Beurling.
Recall that {\it the Hardy space over the infinite polydisc} is the reproducing
kernel Hilbert space associated with the kernel
$$
\mathbb S_\infty(z, w): = \prod_{j=1}^\infty \frac{1}{1 - z_j \overline{w_j}},
\qquad
z = (z_j)_{j=1}^\infty,\ w = (w_j)_{j=1}^\infty \in \mathbb D^\infty \cap \ell^2(\mathbb N).
$$
Due to a celebrated observation of Bohr \cite{Bohr}, a Dirichlet series can be
viewed as a power series in infinitely many variables. Via Bohr's correspondence,
it is well known (see \cite[Section~2.2]{HLS}) that $\mathcal H(\mathbb S_\infty)$ can be identified with
the Hardy space of Dirichlet series, whose reproducing kernel is
$$
(\lambda, u) \mapsto \sum_{n=1}^\infty n^{-\lambda - \overline{u}} = \zeta(\lambda + \overline{u}),
\qquad \lambda, u \in \mathbb H_{\frac{1}{2}},
$$
where $\mathbb H_\rho$ is the half-plane  $\{\, \lambda \in \mathbb C : \text{Re}(\lambda) > \rho \,\}$ and $\zeta$ is
the Riemann zeta function. The difference is that functions in $\mathcal H(\mathbb S_\infty)$ are defined on
$\mathbb D^\infty \cap \ell^2(\mathbb N)$, whereas functions in the Hardy space of Dirichlet series are
defined on the right half-plane $\mathbb H_{1/2}$.
The significance of the Hardy space of Dirichlet series associated with the Riemann
zeta kernel, together with its close connection with $\mathcal H(\mathbb S_\infty)$, motivates the study of
bounded multiplicative linear functionals on the Hardy space of Dirichlet series, and more generally, on
the Hilbert space associated with a diagonal Dirichlet series kernel
$$
k(\lambda, u) = \sum_{n=1}^\infty c_n\, n^{-\lambda - \overline{u}},
\qquad \lambda, u \in \mathbb H_{\rho},
$$
 where $\{c_n\}_{n=1}^\infty$ is a sequence of strictly positive real numbers, and $\rho\in\mathbb{R}$. Deferring the definitions and basic theory of Dirichlet series to ~\cref{Dirichlet_Section}, where they
are explained in context, we state the theorem here.
\begin{theorem} \label{theorem_Main_3}
	Let $\rho \in \mathbb R$, and let $k$ be a diagonal Dirichlet series kernel on $\mathbb H_{\rho}$. Suppose there exists $t > \rho$ such that $k_{t}^2\in \mathcal H(k)$. Then $k_{u}^2\in \mathcal H(k)$ for every $u \in \mathbb H_{t}$. A bounded linear functional  $\Lambda : \mathcal H(k) \to \mathbb C$ 
	satisfying $\Lambda(1) = 1$ is multiplicative if and only if  $\Lambda(k_{u}^2) = \Lambda(k_{u})^2		$ for every $u \in \mathbb H_{t}$.
\end{theorem}

As an application of the preceding characterizations of multiplicative linear functionals, we show that a weighted composition operator can be detected just by checking its actions on certain special functions stemming from the  kernel functions. 

\begin{defn}
A bounded operator $T$ on a reproducing kernel Hilbert space $\mathcal H(k)$ containing the constant functions is a weighted composition operator if there is a holomorphic map $\varphi:\Omega\to\Omega$ satisfying 
		\begin{equation*}
		(Tf)(\bz)=h(\bz)f(\varphi(\bz)),
		\qquad f\in\mathcal H(k),\quad \bz\in\Omega,
			\end{equation*}
where $h \in \mathcal H(k)$ is the function $T(1)$.
\end{defn}

\begin{theorem}
	\label{theorem_application_main}
Let $k$ be a holomorphic kernel on $\Omega$ such that every nonzero multiplicative linear functional on $\mathcal H(k)$ is given by point evaluation at some point of $\Omega$. Let $T \in \mathcal B(\mathcal H(k))$ and suppose that $h:= T(1)$ is nowhere zero on $\Omega$. Then, each of the following three conditions is necessary and sufficient for $T$ to be a weighted composition operator. 
	\begin{enumerate}
		\item[(i)] The kernel k is a diagonal holomorphic kernel, $k_{\bw}^2 \in \mathcal H(k)$ and 
		$T(k_{\bw}^2)\,T(1)=(Tk_{\bw})^2$ for all $\bw\in\Omega$.
	\item[(ii)] The kernel $k$ is a Schur product $s_1\circ\cdots\circ s_p$  where $s_1,\ldots,s_p$ are  normalized, irreducible, diagonal complete Pick kernels and	
	$$T(k_{\bw})\, \prod_{j=1}^{p}T\Big(\frac{1}{s_{j_{\bw}}}\Big) = (T(1))^{p+1} \quad \text{ for all } \bw \in \Omega.$$
 \item[(iii)] The kernel $k$ is a tensor product $s_1\otimes\cdots\otimes s_p$ where $s_1,\ldots,s_p$ are  normalized, irreducible, diagonal complete Pick kernels and
 $$T(k_{\bw})\, \prod_{j=1}^{p}T\Big(\frac{1}{s_{j_{\bw_j}}}\Big) = (T(1))^{p+1} \quad \text{  for all } \bw = (\bw_1,\ldots,\bw_p) \in \Omega \times \cdots \times \Omega$$
 where $k_{\bw}$ is as defined just before \cref{completePick}.
	\end{enumerate}
\end{theorem}

There are several important classes of spaces for which the multiplicative linear functionals under consideration are given by point evaluations. These include, in particular, the Hardy space $H^2(\mathbb{D})$ on the unit disc \cite{mashreghi2015gleason}, the Drury--Arveson spaces $H^2_d$ on the unit ball of $\mathbb C^d$ \cite{Hartz-Can}, for $1\le d\le \infty$, and, more generally, certain complete Pick spaces on the Euclidean ball; see \cite[Corollary 1.3]{AMMcR} together with \cite[Lemma 5.1]{Hartz-Can}. The same conclusion also holds for a class of unitarily invariant kernels under suitable conditions on their coefficient sequences. The diagonal Dirichlet series kernels cannot be included in \cref{theorem_application_main} because unlike the other kernels considered in this note, for a diagonal Dirichlet series kernel, there always are uncountably infinitely many  multiplicative linear functionals which cannot be realized as evaluations at a point in the domain, see  Lemma  \ref{Always_mlf_not_pointev}.

\section{Proof of \cref{theorem_Main_1}} The implication $(1)\Rightarrow(2)$ is immediate, we therefore prove $(2)\Rightarrow(1)$. Since $k$ is a diagonal holomorphic kernel, the expansion
	$$
	k(\bz, \bw)
	=
	\sum_{\alpha\in\mathbb Z_+^d}
	a_\alpha (\bz\odot\overline{\bw})^\alpha 
	$$
	is valid in a neighbourhood of origin and the monomials form an orthogonal basis of
	$\mathcal H(k)$, see \cite[Section~2.3]{MT}. In particular, the polynomials are contained in $\mathcal H(k)$ and are
	dense in $\mathcal H(k)$. 	Henceforth, all calculations will be carried out in a neighborhood of $\mathbf 0$. 	
	Fix $\bw\in\Omega$. The monomial expansion of the kernel function $k_{\bw}$
	is
	$$
	k_{\bw}(\bz)
	=
	\sum_{\alpha\in\mathbb Z_+^d}
	a_\alpha (\bz\odot \overline{\bw})^\alpha 
	=\sum_{\alpha\in\mathbb Z_+^d}
	a_\alpha \bz^\alpha\overline{\bw}^{\,\alpha},
	$$
	where the series converges in $\mathcal H(k)$. Since $\Lambda$ is bounded,
	we apply $\Lambda$ term by term and obtain
	$$
	\Lambda(k_{\bw})
	=
	\sum_{\alpha\in\mathbb Z_+^d}
	a_\alpha\Lambda(\bz^\alpha)\overline{\bw}^{\,\alpha}.
	$$
	Consequently,
\begin{equation*}
	\begin{aligned}
		\bigl[\Lambda(k_{\bw})\bigr]^2
		&=
		\Bigg(
		\sum_{\gamma\in\mathbb{Z}_+^d}
		a_\gamma \Lambda(\bz^\gamma)
		\overline{\bw}^{\,\gamma}
		\Bigg)
		\Bigg(
		\sum_{\eta\in\mathbb{Z}_+^d}
		a_\eta \Lambda(\bz^\eta)
		\overline{\bw}^{\,\eta}
		\Bigg) \\
		&=
		\sum_{\alpha\in\mathbb{Z}_+^d}
		\Bigg(
		\sum_{\gamma\leq\alpha}
		a_\gamma a_{\alpha-\gamma}
		\Lambda(\bz^\gamma)
		\Lambda(\bz^{\alpha-\gamma})
		\Bigg)
		\overline{\bw}^{\,\alpha}.
	\end{aligned}
\end{equation*}
	On the other hand, the pointwise square of $k_{\bw}$ has the power-series
	expansion
	\begin{equation*}
		k_{\bw}(\bz)^2 = 
		\Bigg(
		\sum_{\gamma\in\mathbb Z_+^d}
		a_\gamma\bz^\gamma\overline{\bw}^{\,\gamma}
		\Bigg)
		\Bigg(
		\sum_{\eta\in\mathbb Z_+^d}
		a_\eta\bz^\eta\overline{\bw}^{\,\eta}
		\Bigg)=
		\sum_{\alpha\in\mathbb Z_+^d}
		\Bigg(
		\sum_{\gamma\leq\alpha}
		a_\gamma a_{\alpha-\gamma}
		\Bigg)
		\bz^\alpha\overline{\bw}^{\,\alpha}.
	\end{equation*}
	By hypothesis, $k_{\bw}^2\in\mathcal H(k)$. Hence its monomial expansion
	converges in $\mathcal H(k)$, and the boundedness of $\Lambda$ gives
	$$
	\Lambda(k_{\bw}^2)
	=
	\sum_{\alpha\in\mathbb Z_+^d}
	\Bigg(
	\sum_{\gamma\leq\alpha}
	a_\gamma a_{\alpha-\gamma}
	\Bigg)
	\Lambda(\bz^\alpha)\overline{\bw}^{\,\alpha}.
	$$
	The hypothesis $\Lambda(k_{\bw}^2)=\Lambda(k_{\bw})^2$
	therefore implies
	\begin{equation*}
			\begin{aligned}
		&\sum_{\alpha\in\mathbb Z_+^d}
		\Bigg(
		\sum_{\gamma\leq\alpha}
		a_\gamma a_{\alpha-\gamma}
		\Bigg)
		\Lambda(\bz^\alpha)\overline{\bw}^{\,\alpha}=
		\sum_{\alpha\in\mathbb Z_+^d}
		\Bigg(
		\sum_{\gamma\leq\alpha}
		a_\gamma a_{\alpha-\gamma}
		\Lambda(\bz^\gamma)
		\Lambda(\bz^{\alpha-\gamma})
		\Bigg)
		\overline{\bw}^{\,\alpha}
			\end{aligned}
	\end{equation*}
	for every $\bw$ in a neighbourhood of $\mathbf 0$ in $\Omega$. Thus,
	\begin{equation}\label{eq:coefficient-identity}
		\Bigg(
		\sum_{\gamma\leq\alpha}
		a_\gamma a_{\alpha-\gamma}
		\Bigg)
		\Lambda(\bz^\alpha)
		=
		\sum_{\gamma\leq\alpha}
		a_\gamma a_{\alpha-\gamma}
		\Lambda(\bz^\gamma)
		\Lambda(\bz^{\alpha-\gamma})
	\end{equation}
	for every $\alpha\in\mathbb Z_+^d$. Because $\Lambda(1)=1$, the terms corresponding to
	$\gamma=\mathbf 0$ and $\gamma=\alpha$ on both sides of
	\eqref{eq:coefficient-identity} are equal. Removing these endpoint terms,
	we obtain
	\begin{equation}
		\label{eq:recursive-identity}
		\Bigg(
		\sum_{\substack{\gamma\leq\alpha\\
				\gamma\neq\mathbf 0,\alpha}}
		a_\gamma a_{\alpha-\gamma}
		\Bigg)
		\Lambda(\bz^\alpha) 
		=
		\sum_{\substack{\gamma\leq\alpha\\
				\gamma\neq\mathbf 0,\alpha}}
		a_\gamma a_{\alpha-\gamma}
		\Lambda(\bz^\gamma)
		\Lambda(\bz^{\alpha-\gamma}).
	\end{equation}
	For $\alpha=(\alpha_1,\ldots,\alpha_d)\in\mathbb{Z}_+^d$, set $\Lambda(\bz)^\alpha := \prod_{j=1}^d \Lambda(z_j)^{\alpha_j}.$
	We prove by induction on $|\alpha|$ that
	\begin{equation}\label{eq:monomial-character}
		\Lambda(\bz^\alpha)
		=
		\Lambda(\bz)^\alpha
		\qquad
		\text{for every }\alpha\in\mathbb Z_+^d.
	\end{equation}
	For $|\alpha|=0$, this is precisely the identity $\Lambda(1)=1$. For
	$|\alpha|=1$, the assertion is immediate. Suppose that \eqref{eq:monomial-character} holds for every multi-index of
	degree at most $m$. Let $|\alpha|=m+1$. Note that if
	$\gamma \in \mathbb Z_+^d$ satisfies $\gamma \leq \alpha$ and $\gamma \neq \mathbf 0, \alpha$, then
	$
	1\leq|\gamma|\leq m$
	and $1\leq|\alpha-\gamma|\leq m.
	$
	The induction hypothesis therefore gives
	\begin{equation*}
		\Lambda(\bz^\gamma)
		\Lambda(\bz^{\alpha-\gamma}) =
	\Lambda(\bz)^\gamma \Lambda(\bz)^{\alpha - \gamma} = \Lambda(\bz)^\alpha.
	\end{equation*}
	Thus every summand on the right-hand side of
	\eqref{eq:recursive-identity} contains the same factor
	$
\Lambda(\bz)^\alpha.$ Moreover, the coefficient
	$$
	\sum_{\substack{\gamma\leq\alpha\\
			\gamma\neq\mathbf 0,\alpha}}
	a_\gamma a_{\alpha-\gamma}
	$$
	is strictly positive. 
	Dividing \eqref{eq:recursive-identity} by the
	positive coefficient above, we obtain the required identity \eqref{eq:monomial-character} with $|\alpha| = m+1$. This completes the induction process and proves \eqref{eq:monomial-character}. Consequently, $\Lambda(\bz^{\alpha + \beta}) =  \Lambda(\bz^\alpha) \Lambda(\bz^\beta)$ for all $\alpha, \beta \in \mathbb Z_+^d.$ 
	
	To finish the proof, let 
$$f = \sum_{\alpha \in \mathbb Z_+^d} a_\alpha \mf z^\alpha \text{ and } g = \sum_{\alpha \in \mathbb Z_+^d} b_\alpha \mf z^\alpha$$ 
be in $\mathcal H(k)$ such that $fg \in \mathcal H(k)$. Then, by \eqref{eq:monomial-character} and the continuity of $\Lambda$, we have
\begin{equation}\label{general-entire-space}
	\begin{aligned}
		\Lambda(f)\Lambda(g)  = \Bigg(\sum_{\alpha \in \mathbb Z_+^d} a_\alpha  \Lambda(\mf z^\alpha)\Bigg)\Bigg(\sum_{\beta \in \mathbb Z_+^d} b_\beta \Lambda(\mf z^\beta)\Bigg) 
	&= \sum_{\alpha, \beta \in \mathbb Z_+^d} a_\alpha b_\beta \Lambda(\mf z^\alpha) \Lambda(\mf z^\beta)\\
	&= \sum_{\alpha, \beta \in \mathbb Z_+^d} a_\alpha b_\beta \Lambda\big(\mf z^{\alpha + \beta}\big)\\
	&= \sum_{\alpha \in \mathbb Z_+^d} \Big(\sum_{\beta \leq \alpha} a_\beta b_{\alpha - \beta}\Big) \Lambda(\mf z^{\alpha})\\
	&= \Lambda\Big(\sum_{\alpha \in \mathbb Z_+^d} \Big(\sum_{\beta \leq \alpha} a_\beta b_{\alpha - \beta}\Big) \mf z^{\alpha}\Big) = \Lambda(fg).
	\end{aligned}
\end{equation}	
This completes the proof.	\qed

\section{On the products of complete Pick kernels}\label{section:power}
\subsection{Preliminaries}
We begin with the following characterization. It can be established in essentially the same way as Lemma~2.3 of~\cite{Hartz-Can}, and thus we omit the proof.

\begin{thm}\label{theorem_CNPchar}
	Let $\Omega\subseteq\mathbb C^d$ be a domain containing the origin, and let $k$
	be a diagonal holomorphic kernel on $\Omega$. Then $k$ is a complete Pick kernel if and only if there exist nonnegative
	numbers $\{b_\alpha\}_{\alpha\in\mathbb Z_+^d\setminus\{\mf 0\}}$ with $b_{\alpha}>0$ for $|\alpha|=1$  such that 
	\[
	k(\bz,\bw)
	=
	\frac{1}
	{1-\langle b(\bz),b(\bw)\rangle},
	\]
	in a neighbourhood of $\textbf{0}$ where
	$b(\bz)
	=
	\bigl(b_\alpha \bz^\alpha\bigr)_{\alpha\in\mathbb Z_+^d\setminus\{\mf 0\}}
	\in\ell^2(\mathbb Z_+^d\setminus\{\mf 0\}).
	$
\end{thm}
  We also need an elementary lemma whose proof we omit because it follows from the fact that $b^{(m)}$ is a contractive (row) multiplier where 
  \begin{equation}\label{equation:pair1}
  	s_m(\bz, \bw) = \frac{1}{1-\langle b^{(m)}(\bz), b^{(m)}(\bw)\rangle},
  \end{equation} 
  see Theorem 8.2 of \cite{AMbook}.

\begin{lemma}\label{lemma_ker}
	Let $s_1, \ldots, s_p$ be complete Pick kernels on $\Omega$ and let $k = s_1 \circ \cdots \circ s_p$, the Schur product. For every $\bw$, the kernel function $s_{{j}_{\bw}}^{-1}$ is in $M(\mathcal{H}(k))$ for all $j =1, \ldots, p$.  
\end{lemma}

\subsection{Proof of part (1) of \cref{completePick}:} 

The forward implication is immediate from the multiplicativity of $\Lambda$ and the fact that $s_{j_{\bw}}^{-1}$ belongs to the multiplier algebra and hence in the space for every $\bw\in\Omega$, by ~\cref{lemma_ker}. We shall only prove the reverse implication. 

For notational convenience, we prove the result only for $p=2$. Since the kernels $s_1$ and $s_2$ are of the form as in \eqref{equation:pair1}, we have 
\begin{equation*}
	\begin{aligned}
		\Lambda(k_{\bw})&=\frac{1}{\Lambda(\frac{1}{s_{1_{\bw}}})\Lambda(\frac{1}{s_{2_{\bw}}})}\\
		&=\frac{1}{\Lambda\left(1-\langle b^{(1)}(\bz), b^{(1)}(\bw)\rangle\right)\Lambda\left(1-\langle b^{(2)}(\bz), b^{(2)}(\bw)\rangle\right)}\\
		&= \frac{1}{\left(1-\Lambda[\langle b^{(1)}(\bz), b^{(1)}(\bw)\rangle]\right)\left(1-\Lambda\left[\langle b^{(2)}(\bz), b^{(2)} (\bw)\rangle\right]\right)}. 
	\end{aligned}
\end{equation*}
Observe that 
$\langle b^{(j)}(\bz),b^{(j)}(\mathbf 0)\rangle=\mathbf O,$ for $j=1,2,$
where $\mathbf O$ denotes the zero function. Considering the two families of functions continuously indexed by $\bw$ whose values at $\bz$ are $\langle b^{(j)}(\bz),b^{(j)}(\mathbf w)\rangle$ for $j=1,2$ and noting that  $\Lambda$ is continuous and $
\Lambda(\mathbf O)=0,$
it follows that there exists $\varepsilon>0$ such that
\[ \bigl|\Lambda(\langle b^{(1)}(\bz),b^{(1)}(\bw)\rangle)\bigr|<1
\quad\text{and}\quad
\bigl|\Lambda(\langle b^{(2)}(\bz),b^{(2)}(\bw)\rangle)\bigr|<1
\]
for every $\bw$ in a neighbourhood of the origin which we shall call $ X(\mathbf 0,\varepsilon)$. Throughout the remainder of the proof, all computations will be carried out with $\bw$ restricted to the neighbourhood $X(\mathbf 0,\varepsilon)$.  For $j=1,2$, fix a point $\bw$ and set
\[
B_{\bw}^{(j)}(\bz)
:=
\Lambda\bigl(\langle b^{(j)}(\bz),b^{(j)}(\bw)\rangle\bigr)
=
\sum_{\alpha\in\mathbb Z_+^d\setminus\{\mathbf0\}}
b_\alpha^{(j)}\Lambda(\bz^\alpha)\overline{\bw}^{\,\alpha}.
\]
Since
\[
\Lambda(k_{\bw})
=
\left(
\sum_{n=0}^\infty (B_{\bw}^{(1)}(\bz))^n
\right)
\left(
\sum_{m=0}^\infty (B_{\bw}^{(2)}(\bz))^m
\right),
\]
expanding each geometric series and taking the Cauchy product yields
\begin{equation}\label{equation16}
	\begin{aligned}
		\Lambda(k_{\bw})
		&=
		\sum_{\alpha\in\mathbb Z_+^d}
		\Biggl[
		\sum_{\delta_1+\delta_2=\alpha}
		C_{\delta_1}^{(1)}(\bz)
		C_{\delta_2}^{(2)}(\bz)
		\Biggr]
		\overline{\bw}^{\,\alpha},
	\end{aligned}
\end{equation}
where $C_{\mathbf0}^{(j)}(\bz):=1,$ and for a multi-index $\delta\neq\mathbf0$,
\[
C_\delta^{(j)}(\bz)
=
\sum_{q=1}^{|\delta|}
\;
\sum_{\substack{
		\gamma_1,\ldots,\gamma_q\in
		\mathbb Z_+^d\setminus\{\mathbf0\}\\
		\gamma_1+\cdots+\gamma_q=\delta}}
b_{\gamma_1}^{(j)}\cdots b_{\gamma_q}^{(j)}
\Lambda(\bz^{\gamma_1})\cdots
\Lambda(\bz^{\gamma_q}).
\]
On the other hand, since $k_{\bw}=s_{1_{\bw}} \circ s_{2_{\bw}},$ we obtain
\begin{equation}\label{equation17}
	\begin{aligned}
		\Lambda(k_{\bw})
		&=
		\Lambda\!\left[
		\frac1{1-\langle b^{(1)}(\bz),b^{(1)}(\bw)\rangle}
		\frac1{1-\langle b^{(2)}(\bz),b^{(2)}(\bw)\rangle}
		\right]                                    \\
		&=
		\sum_{\alpha\in\mathbb Z_+^d}
		\Biggl[
		\sum_{\delta_1+\delta_2=\alpha}
		c_{\delta_1}^{(1)}
		c_{\delta_2}^{(2)}
		\Biggr]
		\Lambda(\bz^\alpha)
		\overline{\bw}^{\,\alpha},
	\end{aligned}
\end{equation}
where $
c_{\mathbf0}^{(j)}:=1,$ and, for a multi-index $\delta\neq\mathbf0$,
\[
c_\delta^{(j)}
=
\sum_{q=1}^{|\delta|}
\;
\sum_{\substack{
		\gamma_1,\ldots,\gamma_q\in
		\mathbb Z_+^d\setminus\{\mathbf0\}\\
		\gamma_1+\cdots+\gamma_q=\delta}}
b_{\gamma_1}^{(j)}\cdots b_{\gamma_q}^{(j)}.
\]
Since both sides are analytic in $\bw$, the corresponding power series in $\overline{\bw}$ agree on $X(0,\varepsilon)$. Comparing the coefficients of
$\overline{\bw}^{\,\alpha}$ in \eqref{equation16} and \eqref{equation17}, we conclude that
\[
\sum_{\delta_1+\delta_2=\alpha}
C_{\delta_1}^{(1)}(\bz)
C_{\delta_2}^{(2)}(\bz)
=
\left(
\sum_{\delta_1+\delta_2=\alpha}
c_{\delta_1}^{(1)}
c_{\delta_2}^{(2)}
\right)
\Lambda(\bz^\alpha),
\qquad
\alpha\in\mathbb Z_+^d.
\]
Substituting the definitions of $C_{\delta}^{(j)}$ and $c_{\delta}^{(j)}$, we obtain
\begin{equation}\label{equation18}
	\begin{aligned}
		&
		\sum_{\delta_1+\delta_2=\alpha}
		\Biggl(
		\sum_{r=1}^{|\delta_1|}
		\sum_{\substack{
				\gamma_1+\cdots+\gamma_r=\delta_1}}
		b_{\gamma_1}^{(1)}\cdots b_{\gamma_r}^{(1)}
		\Lambda(\bz^{\gamma_1})\cdots
		\Lambda(\bz^{\gamma_r})
		\Biggr)                                   \\
		&\qquad\qquad\qquad\times
		\Biggl(
		\sum_{q=1}^{|\delta_2|}
		\sum_{\substack{
				\eta_1+\cdots+\eta_q=\delta_2}}
		b_{\eta_1}^{(2)}\cdots b_{\eta_q}^{(2)}
		\Lambda(\bz^{\eta_1})\cdots
		\Lambda(\bz^{\eta_q})
		\Biggr)                                   \\
		&=
		\left(
		\sum_{\delta_1+\delta_2=\alpha}
		\Biggl(
		\sum_{r=1}^{|\delta_1|}
		\sum_{\substack{
				\gamma_1+\cdots+\gamma_r=\delta_1}}
		b_{\gamma_1}^{(1)}\cdots b_{\gamma_r}^{(1)}
		\Biggr)
		\Biggl(
		\sum_{q=1}^{|\delta_2|}
		\sum_{\substack{
				\eta_1+\cdots+\eta_q=\delta_2}}
		b_{\eta_1}^{(2)}\cdots b_{\eta_q}^{(2)}
		\Biggr)
		\right)
		\Lambda(\bz^\alpha),
	\end{aligned}
\end{equation}
for every $\alpha\in\mathbb Z_+^d\setminus\{\mathbf0\}$.

\qquad

\textbf{Claim:}
$\Lambda(\bz^{\alpha})=\Lambda(\bz)^{\alpha}$, for all  $\alpha =(\alpha_1,\ldots,\alpha_d)\in \mathbb{Z}_+^d$. 

\quad

The result is immediate for $\alpha = 0$ and for all $\alpha$ with $|\alpha| = 1$; for the general case, we proceed by induction. Let $\alpha = (\alpha_1,\ldots,\alpha_d)$ and suppose that $|\alpha| = 2$. 
Then there are two possibilities: either $\alpha_i = 2$ for some $i$, or 
$\alpha_i = \alpha_j = 1$ for some $i \neq j$. 
First, assume that the entry $2$ occurs in the $i$th position. Then we have
$$
\delta_1 = (0,\ldots,1,\ldots,0)=\delta_2
$$
where the entry $1$ is in the $i$-th position. For this choice of $\alpha$, after cancelling the common terms corresponding to the one-term decompositions of $\alpha$,  \cref{equation18}  reduces to
\begin{equation*}
	b_{\gamma_1}^{(1)}\Lambda(z_i)b_{\eta_1}^{(2)}\Lambda(z_i)=b_{\gamma_1}^{(1)}b_{\eta_1}^{(2)}\Lambda(z_iz_i).
\end{equation*}
Since $b^{(j)}_\alpha>0$ for all multi-indices $\alpha$ with $|\alpha|=1,$ it follows that $\Lambda(z_i^2)=\Lambda(z_i)^2$. Now suppose there exist distinct indices $i \neq j$ such that $\alpha_i = \alpha_j = 1$. In this case, \cref{equation18} collapses to 
$$\sum _{\delta_1+\delta_2=\alpha}b_{\delta_1}^{(1)} b_{\delta_1}^{(2)}\Lambda(z_j)\Lambda(z_i)=\sum _{\delta_1+\delta_2=\alpha}b_{\delta_2}^{(1)}b_{\delta_2}^{(2)}\Lambda(z_jz_i).$$
Again, by the same reasoning, these sums are non-zero and therefore cancel out, yielding $\Lambda(z_j z_i) = \Lambda(z_j)\Lambda(z_i)$. 
Hence, $\Lambda(\bz^\alpha) = \Lambda(\bz)^\alpha$ for all multi-indices $\alpha$ with $|\alpha| = 2$.

Suppose the claim is valid for all $\alpha$ satisfying $|\alpha|\le k$. It suffices to prove the claim for $|\alpha|=k+1$, after which the conclusion follows by induction. 
To this end, suppose we have a $\alpha$ with $\mod{\alpha}=k+1$. As $|\delta_1+\delta_2|=k+1$ and $\mod{\delta_i}\le k$ as $\delta_i\neq 0$, this implies that $\mod{\gamma_i}\le k$ and $\mod{\eta_j}\le k$. Thus, by applying the induction hypothesis for this fixed multi-index $\alpha$, \cref{equation18} simplifies to

\begin{equation*}
	\begin{aligned}
		&\left[\sum_{\delta_1+\delta_2=\alpha}\left(\sum_{1\le r\le \mod{\delta_1}}\sum_{\gamma_1+\dots+\gamma_r=\delta_1}b_{\gamma_1}^{(1)}\dots b_{\gamma_r}^{(1)}\right)\left(\sum_{1\le r\le \mod{\delta_2}}\sum_{\eta_1+\dots+\eta_q=\delta_2}b_{\eta_1}^{(2)}\dots b_{\eta_q}^{(2)}\right)\right]\Lambda(\bz)^{\alpha}\\
		&=\left[
		\sum_{\delta_1+\delta_2=\alpha}
		\left(\sum_{1\le r\le \mod{\delta_1}}
		\sum_{\substack{\gamma_1+\cdots+\gamma_r = \delta_1}}
		b_{\gamma_1}^{(1)} \cdots b_{\gamma_r}^{(1)}
		\right)
		\left(\sum_{1\le r\le \mod{\delta_2}}
		\sum_{\substack{\eta_1+\cdots+\eta_q = \delta_2}}
		b_{\eta_1}^{(2)} \cdots b_{\eta_q}^{(2)}
		\right)\right]
		\Lambda(\bz^{\,\alpha}).
	\end{aligned}
\end{equation*}
Since each $b^j_\alpha$ is nonnegative for $j=1,2$ and $\alpha$, and $b_\alpha = 1$ whenever $|\alpha|=1$, the sum is always nonnegative and contains at least one term equal to $b_1$. Hence, the sum is strictly positive. Consequently, the same nonzero term appears on both sides, which readily implies that $\Lambda(\bz^{\alpha}) = \Lambda(\bz)^{\alpha}$. This proves the claim. 

For the remaining part, we follow the arguments as in  \eqref{general-entire-space} and conclude that $\Lambda$ is multiplicative on the entire space. This completes the proof. 
\qed

The generalized Bergman kernels 
$$\left(\frac{1}{1 - \langle \bz, \bw \rangle}\right)^m$$ 
for $m \in \mathbb N$ as well as any positive integral power of the Dirichlet kernel provide a prime source of examples.

It is not difficult to verify that the hypotheses of the theorem are satisfied when $s_j=s$ for all $j$. Although this is well known, we include a proof for the convenience of the reader.
\begin{lemma}\label{Lemma_SP}
	Let $s$ be a complete Pick kernel on $\Omega$, and let $
	k$ be the Schur power $k=s^{\circ p},$ where $p\in\mathbb N$. Then, for every $\bw\in \Omega$, the function
	$k_{\bw}^2\in M(\mathcal H(k)).$
\end{lemma}

\begin{proof}
	By Theorem~\ref{theorem_CNPchar}, there exist a Hilbert space $\mathcal E$ and a function
	$b$ on $\Omega$ such that
	\[
	s(\bz,\bw)
	=
	\frac{1}{1-\langle b(\bz),b(\bw)\rangle_{\mathcal E}}.
	\]
	By \cite[Theorem~8.2]{AMbook}, $b$ is a contractive row multiplier. Hence, the function
	\[
	\varphi_{\bw}(\bz)
	=
	\langle b(\bz),b(\bw)\rangle_{\mathcal E}, \quad \bz \in \Omega
	\]
	belongs to ${M}(\mathcal H(k))$ for each
	$\bw\in\Omega$, and satisfies
	$
	\|M_{\varphi_{\bw}}\|
	\leq
	\|b(\bw)\|_{\mathcal E}
	<1.
	$
	Thus, $1-\varphi_{\bw}$ is invertible in ${M}(\mathcal H(k))$, and therefore $s_{\bw} \in 	M(\mathcal H(k)).$
	Since $M(\mathcal H(k))$ is an algebra, $k_{\bw}^2 	= 	s_{\bw}^{2p} \in M(\mathcal H(k)).$
\end{proof}

\subsection{Proof of part (2) of \cref{completePick}:} 

Adopting a similar strategy as before, it is straightforward to see that for a fixed $\boldsymbol t$, the functions $
s_{j_{\boldsymbol t}}^{-1}$ , for $j=1, \cdots, p$ are in the multiplier algebra of $\mathcal H(k)$. We only prove the converse implication. Again, for notational convenience, we prove the result only for $p=2$. 

Fix $(\by, \bt)\in\Omega\times\Omega$. Keeping in mind that we are now working with $2d$ variables, namely $x_1, \ldots, x_d$ and $v_1, \ldots, v_d$, it suffices to prove that $\Lambda$ is multiplicative on the monomials. Set 
	$$
	\{ \bx^\alpha \bv^\beta : \alpha, \beta \in \mathbb{Z}_+^d \}.$$
	Once this is established, the remainder of the argument follows as before. To this end, since $s_j$ has the form \eqref{equation:pair1} for $j=1,2$, there exists $\varepsilon>0$ such that
$\bigl|\Lambda(\langle b^{(j)}(\bz),b^{(j)}(\bw)\rangle)\bigr|<1$
	for all $\bw$ in a neighbourhood of the origin, $X(0,\varepsilon)$. Consequently, the geometric series converges absolutely, and hence
	$$
	\Lambda\big(k_{(\by,\bt)}\big)
	=
	\left(
	\sum_{n=0}^{\infty}
	\left(
	\Lambda\!\left[ \langle b^{(1)}(\bx), b^{(1)}(\by)\rangle \right]
	\right)^n
	\right)
	\cdot
	\left(
	\sum_{m=0}^{\infty}
	\left(
	\Lambda\!\left[ \langle b^{(2)}(\bv), b^{(2)}(\bt)\rangle \right]
	\right)^m
	\right).
	$$
	We now simplify the above expression. The right-hand side becomes
	\begin{eqnarray*}
			& & \sum_{n=0}^\infty\left(\Lambda\left[\sum_{\alpha}b^{(1)}_{\alpha}\bx^{\alpha}\overline{\by}^\alpha\right]\right)^n \cdot \sum_{m=0}^\infty\left(\Lambda\left[\sum_{\beta}b_\beta^{(2)}\bv^{\beta}\overline{\bt}^\beta\right]\right)^m \\
			&= & \sum_{n=0}^\infty\left(\sum_{\alpha}b^{(1)}_{\alpha}\Lambda(\bx^{\alpha})\overline{\by}^\alpha\right)^n \cdot \sum_{m=0}^\infty\left(\sum_{\beta}b_\beta^{(2)}\Lambda(\bv^{\beta})\overline{\bt}^\beta\right)^m.
				\end{eqnarray*}
			Using multinomial expansion and collecting like powers, we see that the coefficient of $\overline{\by}^{\alpha}\,\overline{\bt}^{\beta}$ 
	in the expansion of $\Lambda\big(k_{(\by,\bt)}\big)$ is 
	\begin{equation*}
		\sum_{1 \le r \le |\alpha|}
		\sum_{1 \le q \le |\beta|}
		\sum_{\gamma_1+\cdots+\gamma_r=\alpha}
		\sum_{\eta_1+\cdots+\eta_q=\beta}
		b_{\gamma_1}^{(1)} \cdots b_{\gamma_r}^{(1)}\,
		b_{\eta_1}^{(2)} \cdots b_{\eta_q}^{(2)}\,
		\Lambda\!\bigl(\bx^{\gamma_1}\bigr)\cdots
		\Lambda\!\bigl(\bx^{\gamma_r}\bigr)\,
		\Lambda\!\bigl(\bv^{\eta_1}\bigr)\cdots
		\Lambda\!\bigl(\bv^{\eta_q}\bigr).
	\end{equation*}
	On the other hand, since $k_{(\by, \bt)}= s_{1_{\by}}s_{2_{\bt}}$, we get that the coefficient of $\overline{\by}^{\alpha}\,\overline{\bt}^{\beta}$ in the expansion of $\Lambda(k_{(\by, \bt)})$ is 
$$\sum_{\alpha}\sum_{\beta}\Bigg[\left(\sum_{1\le r\le \mod{\alpha}}\sum_{\gamma_1+\dots+\gamma_r=\alpha}b_{\gamma_1}^{(1)}\dots b_{\gamma_r}^{(1)} \Lambda(\bx^{\alpha}) \right)\left(\sum_{1\le r\le \mod{\beta}}\sum_{\eta_1+\dots+\eta_q=\beta}b_{\eta_1}^{(2)}\dots b_{\eta_q}^{(2)} \Lambda(\bv^\beta) \right)\Bigg]. $$
	Equating them we have
	\begin{equation}\label{coefficientcompare}
		\begin{aligned}
			&\sum_{1\le r\le |\alpha|}
			\sum_{1\le q\le |\beta|}
			\sum_{\gamma_1+\cdots+\gamma_r=\alpha}
			\sum_{\eta_1+\cdots+\eta_q=\beta}
			b_{\gamma_1}^{(1)}\cdots b_{\gamma_r}^{(1)}
			b_{\eta_1}^{(2)} \cdots b_{\eta_q}^{(2)} \Lambda\!\bigl(\bx^{\gamma_1}\bigr)\cdots
			\Lambda\!\bigl(\bx^{\gamma_r}\bigr)\,
			\Lambda\!\bigl(\bv^{\eta_1}\bigr)\cdots
			\Lambda\!\bigl(\bv^{\eta_q}\bigr)\\
			&=\sum_{1\le r\le |\alpha|}
			\sum_{1\le q\le |\beta|}
			\sum_{\gamma_1+\cdots+\gamma_r=\alpha}
			\sum_{\eta_1+\cdots+\eta_q=\beta}
			b_{\gamma_1}^{(1)} \cdots b_{\gamma_r}^{(1)}
			b_{\eta_1}^{(2)} \cdots b_{\eta_q}^{(2)} \Lambda\!\bigl(\bx^{\gamma_1}\cdots \bx^{\gamma_r}\bigr)\,
			\Lambda\!\bigl(\bv^{\eta_1}\cdots
			\bv^{\eta_q}\bigr). 
		\end{aligned} 
	\end{equation}
	 Recall that we use the shorthand
	\begin{equation*}
		\Lambda(\bz)^{\alpha} := \prod_{j=1}^d \Lambda(z_j)^{\alpha_j},
	\end{equation*}
	for $\bz = (z_1, \ldots, z_d) \in \Omega$ and $\alpha = (\alpha_1, \ldots, \alpha_d) \in \mathbb{Z}_+^d$. 
	
	We now apply an induction argument on $|\alpha|+|\beta|$, the details of which are omitted, as they are entirely analogous to those in the proof of part~(1). Using \cref{coefficientcompare}, the non-negativity of $b_\alpha^{(1)}$ and $b_\beta^{(2)}$ for all multi-indices $\alpha,\beta$, and their strict positivity when $|\alpha|=1$ and $|\beta|=1$, respectively, we obtain
$$\Lambda(\bx^\alpha \bv^\beta)=\Lambda(\bx)^\alpha\Lambda( \bv)^\beta  \qquad \text{for all multi-indices $\alpha$ and $\beta$}.$$  
The remaining part is now standard. This completes the proof. 
\qed

 	It is worth noting that the characterization of multiplicativity of linear functionals obtained for the Dirichlet space in \cite{mashreghiDirichlet}, as well as for complete Pick spaces in \cite{CHMR}, relies on certain nontrivial factorization results developed in \cite{AP,DP} and \cite{AMMcR} for the Dirichlet space, and in \cite{JM}  for complete Pick spaces. The applicability of the result in \cite{JM} is due to the landmark work of Hartz \cite{Hartz-Acta}, where it is shown that the multiplier algebra of complete Pick kernels satisfies the column–row property. However, our proofs are by elementary methods.

\section{On the spaces of Dirichlet series} \label{Dirichlet_Section}

\subsection{A primer on Dirichlet series} A {\it Dirichlet series} is a series of the form
$$
f(\lambda) = \sum_{n=1}^{\infty} a_{n} n^{-\lambda},
$$
where $a_n \in \mathbb C,~n \ge 1$ and $\lambda$ is a complex variable.
If $a_n = 1$ for all $n \ge 1,$ then we have the Riemann zeta function.
If $f$ is convergent at $\lambda=\lambda_{0}$, then it converges uniformly throughout the angular region
$$
\left\{\lambda \in \mathbb C : |\arg(\lambda-\lambda_0)| < \frac{\pi}{2} -\delta\right\}
$$
for every positive real number $\delta < \frac{\pi}{2}.$ Consequently, $f$ defines a holomorphic function on the right half-plane
$
\{\lambda \in \mathbb C : \operatorname{Re}(\lambda) > \operatorname{Re}(\lambda_0)\}
$
(we refer to \cite[Chapter~IX]{Ti} for the basic theory of Dirichlet series). 

The abscissa of {\it absolute convergence} $\sigma_{a}(f)$ of $f$ is the extended real number defined by
\begin{equation*}
	\sigma_{a}(f)
	=
	\inf\Big\{
	\sigma:
	\sum_{n=1}^{\infty} |a_n| n^{-\sigma}<\infty
	\Big\}.
\end{equation*}
Recall that for a real number $\rho$, the right half-plane
$
\{\lambda \in \mathbb C : \operatorname{Re}(\lambda)>\rho\}
$
is denoted by $\mathbb H_\rho$. The following result asserts that a Dirichlet series is uniquely determined by its values on any right half-plane where it converges (see, \cite[Section~9.6]{Ti}).

\begin{prop}
	For $\rho \in \mathbb R,$ suppose that the Dirichlet series
	$
	f(\lambda)=\sum_{n=1}^\infty a_n n^{-\lambda}
	$
	converges on $\mathbb H_\rho$. If $f=0$ on $\mathbb H_\rho$, then $a_n=0$ for every integer $n\ge1.$
\end{prop}

The next proposition describes the product of two Dirichlet series (see \cite[Theorem~4.3.1 and discussion prior to Theorem~4.3.4]{QQ}).

\begin{prop}
	Let
	$
	f(\lambda)=\sum_{n=1}^\infty a_n n^{-\lambda}~
	\text{and}~
	g(\lambda)=\sum_{n=1}^\infty b_n n^{-\lambda}
	$
	be two convergent Dirichlet series on $\mathbb H_\rho$, for some $\rho\in\mathbb R$. If $g$ converges absolutely on $\mathbb H_\rho,$ then the product is given by
	$$
	f(\lambda)g(\lambda)
	=
	\sum_{n=1}^\infty
	\Big(
	\sum_{d\mid n}
	a_d b_{\frac{n}{d}}
	\Big)
	n^{-\lambda},
	$$
	and it converges on $\mathbb H_\rho.$
\end{prop}

A \emph{diagonal Dirichlet series kernel} is a positive definite kernel of the form
\begin{equation}
	\label{kappa-d}
	k(\lambda,u)
	=
	\sum_{n=1}^\infty
	c_n n^{-\lambda-\overline{u}},
	\qquad
	\lambda,u\in
	\mathbb H_{\rho},
\end{equation}
where $c_n>0$ for every $n$, and
$\rho$ is chosen so that the series converges on $\mathbb H_\rho\times\mathbb H_\rho$
(equivalently, $2\rho$ is bigger than or equal to $\sigma_a(\sum_{n=1}^\infty c_n\, n^{-\lambda})$). 
Since $k$ converges absolutely on
$
\mathbb H_{\rho}
\times
\mathbb H_{\rho},
$
there exists a unique reproducing kernel Hilbert space $\mathcal H(k)$ associated with $k$, given by
$$
\mathcal H(k)
=
\left\{
\sum_{n=1}^\infty a_n n^{-\lambda}
:
\sum_{n=1}^\infty
\frac{|a_n|^2}{c_n}
<
\infty
\right\}.
$$
By the Cauchy-Schwarz inequality, every function in $\mathcal H(k)$ converges absolutely on $\mathbb H_{\rho}$.

\subsection{Proof of \cref{theorem_Main_3}}
 For a Dirichlet series kernel $k$ as in \eqref{kappa-d}, observe that
	\[
	k(\lambda,u)
	=
	(k_u)(\lambda)
	=
	\sum_{n=1}^{\infty}
	\left(c_n n^{-\overline{u}}\right)n^{-\lambda},
	\qquad
	\lambda,u\in\mathbb H_{\rho}.
	\]
	Thus, for every $u\in\mathbb H_{\rho}$, the product formula for Dirichlet series yields
	\begin{equation}
		\begin{aligned}
			k(\lambda,u)^2
			&=
			\left(
			\sum_{n=1}^{\infty}
			\left(c_n n^{-\overline{u}}\right)n^{-\lambda}
			\right)
			\left(
			\sum_{m=1}^{\infty}
			\left(c_m m^{-\overline{u}}\right)m^{-\lambda}
			\right)\\
			&=
			\sum_{n=1}^{\infty}
			\Bigg(
			\sum_{d\mid n}
			c_d d^{-\overline{u}}
			\,c_{n/d}
			\left(\frac{n}{d}\right)^{-\overline{u}}
			\Bigg)
			n^{-\lambda}\\
			&=
			\sum_{n=1}^{\infty}
			\Bigg(
			n^{-\overline{u}}
			\sum_{d\mid n}
			c_d c_{n/d}
			\Bigg)
			n^{-\lambda},
			\qquad
			\lambda\in\mathbb H_{\rho}.
		\end{aligned}
		\label{square-pdt-formula-dirichlet}
	\end{equation}
	Since there exists $t > \rho$ such that $
	k_t^2\in\mathcal H(k),
	$ we have for every $u\in\mathbb H_t$ that
	\begin{equation*}
		\begin{aligned}
			\sum_{n=1}^{\infty}
			\frac{
				\Big(
				\sum\limits_{d\mid n}
				c_d c_{n/d}
				\Big)^2
				n^{-2\operatorname{Re}(u)}
			}
			{c_n}
			\le
			\sum_{n=1}^{\infty}
			\frac{
				\Big(
				\sum\limits_{d\mid n}
				c_d c_{n/d}
				\Big)^2
				n^{-2t}
			}
			{c_n}
			=
			\left\|k_t^2\right\|_{\mathcal H(k)}^2
			<\infty,
		\end{aligned}
	\end{equation*}
	where the equality follows from \eqref{square-pdt-formula-dirichlet}. Therefore, $
	k_u^2
	\in
	\mathcal H(k)
	$ for every $u\in\mathbb H_t$. 
	
	For the remaining equivalence, the forward implication is trivially true. For the converse part, suppose that
	\begin{equation*}
		\Lambda(k_{u}^2) = \Lambda(k_{u})^2~\text{for every $u \in \mathbb H_{t}$}.
	\end{equation*} 
Expanding the square and applying the uniqueness of Dirichlet series yields that
	\[
	\Big(\sum_{d | n} c_d c_{\frac{n}{d}}\Big)\Lambda(n^{-\lambda})
	=
	\sum_{d\mid n}
	c_{d} c_{\frac{n}{d}}\Lambda(d^{-\lambda})
	\Lambda\left(
	\left(
	\frac{n}{d}
	\right)^{-\lambda}
	\right),
	\quad
	n \ge 1.
	\]
	Hence, by using $\Lambda(1) = 1$, we get the recursive formula of the form
	\begin{equation}
		\label{recursive-formula}
		\Lambda(n^{-\lambda})
		=
		\Bigg(\sum_{\substack{j | n \\ j \neq 1, n}} c_j c_{\frac{n}{j}}\Bigg)^{-1}
		\sum_{\substack{d | n \\ d \neq 1, n}}
		c_{d} c_{\frac{n}{d}}
		\Lambda(d^{-\lambda})
		\Lambda\left(
		\left(
		\frac{n}{d}
		\right)^{-\lambda}
		\right)
	\end{equation}
	for all $n \geq 1$ which is not a prime. 
	For obtaining the multiplicativity of $\Lambda$, we first verify that $\Lambda$ is multiplicative on the Dirichlet monomials; i.e.,
	\begin{equation}
		\label{multiplicativity-monomials-product}
		\Lambda\big((mn)^{-\lambda}\big)
		=
		\Lambda(m^{-\lambda})
		\Lambda(n^{-\lambda}),
		\quad
		m,n \ge 1.
	\end{equation}
	To prove this, it suffices to establish that for a fixed $k \in \mathbb N$ and primes $\{p_{m_1}, \ldots, p_{m_k}\}$,  the following holds:
	\begin{equation}
		\label{multiplicative-monomials}
		\noindent \textbf{Claim:} ~~\Lambda(n^{-\lambda}) = \Lambda(p_{m_1}^{-\lambda})^{\alpha_1} \Lambda(p_{m_2}^{-\lambda})^{\alpha_2} \cdots \Lambda(p_{m_k}^{-\lambda})^{\alpha_k} \quad \text{whenever}~ n
		=
		p_{m_1}^{\alpha_1}
		p_{m_2}^{\alpha_2}
		\cdots
		p_{m_k}^{\alpha_k}.
	\end{equation}
	For notational convenience, we write $q_i$ for $p_{m_i}$ for $i=1, \ldots, k$.  We apply the induction argument on $\sum_{i=1}^{k}\alpha_i$. If $\sum_{i=1}^{k}\alpha_i$ is either $0$ or $1$, then there is nothing to prove. If
	$\sum_{i=1}^{k}\alpha_i=2$, then $\alpha = (\alpha_1, \ldots, \alpha_k)$ is either $e_i+e_j$ for some distinct $i, j$, or $2 e_i$ for some $i$. When $\alpha = e_i+e_j$ for some distinct $i, j$, then we need to show that
	$$
	\Lambda((q_iq_j)^{-\lambda})
	=
	\Lambda(q_i^{-\lambda})
	\Lambda(q_j^{-\lambda}).
	$$
	By \eqref{recursive-formula}, note that
	$$
	\Lambda((q_iq_j)^{-\lambda})
	=
	(2c_{q_i}c_{q_j})^{-1}
	\left(
	c_{q_i}c_{q_j}\Lambda(q_i^{-\lambda})\Lambda(q_j^{-\lambda})
	+
	c_{q_j}c_{q_i}\Lambda(q_j^{-\lambda})\Lambda(q_i^{-\lambda})
	\right)
	=
	\Lambda(q_i^{-\lambda})
	\Lambda(q_j^{-\lambda}).
	$$
	In the other case $\alpha = 2e_i$, we have to show that
	$
	\Lambda(q_i^{-2\lambda})
	=
	\Lambda(q_i^{-\lambda})^2.
	$
	Again, by \eqref{recursive-formula}, 
	$$
	\Lambda(q_i^{-2\lambda})
	=
	(c_{q_i}^2)^{-1}
	\left(c_{q_i}^2\Lambda(q_i^{-\lambda})^2\right)
	=
	\Lambda(q_i^{-\lambda})^2.
	$$
	This proves the base case.
	
	 Now, assume for all $(\beta_1,\ldots,\beta_k)\in\mathbb{Z}_+^k$ satisfying
	$
	\sum_{j=1}^{k}\beta_j
	\leq m
	$
	that
	\begin{equation}
		\label{induction-hypothesis}
		\Lambda\left(
		(q_1^{\beta_1}
		q_2^{\beta_2}
		\cdots
		q_k^{\beta_k})^{-\lambda}
		\right)
		=
		\Lambda(q_1^{-\lambda})^{\beta_1}
		\cdots
		\Lambda(q_k^{-\lambda})^{\beta_k}.
	\end{equation}
	Let $\alpha = (\alpha_1,\ldots,\alpha_k)\in\mathbb{Z}_+^k$ be such that
	$
	\sum_{j=1}^{k}\alpha_j
	=
	m+1.
	$
	We sometimes abbreviate $\prod_{j=1}^k q_j^{\gamma_j}$ as $\mathbf{q}^{\gamma}$ for $\gamma = (\gamma_1, \ldots, \gamma_k) \in \mathbb Z_+^k$. Since $|\beta| \leq m$ for any $\beta \leq \alpha$ with $\beta  \neq \alpha$, by the induction hypothesis \eqref{induction-hypothesis},
	\[
	\sum_{\substack{ \beta \le \alpha \\ \beta \neq \mathbf 0, \alpha}}
	c_{\mathbf q^\beta} c_{\mathbf q^{\alpha - \beta}}
	\Lambda\Bigg(\Big(\prod_{j=1}^k q_j^{\beta_j}\Big)^{-\lambda}\Bigg)
	\Lambda\Bigg(\Big(\prod_{j=1}^k q_j^{\alpha_j - \beta_j}\Big)^{-\lambda}\Bigg)
	=
	\prod_{j=1}^k \Lambda(q_j^{-\lambda})^{\alpha_j}
	\sum_{\substack{ \beta \le \alpha \\ \beta \neq \mathbf 0, \alpha}}
	c_{\mathbf q^{\beta}} c_{\mathbf q^{\alpha- \beta}}.
	\]
	Hence, by using \eqref{recursive-formula},
	\begin{eqnarray*}
		\Lambda
		\Bigg(\prod_{j=1}^k q_j^{-\lambda \alpha_j}\Bigg)
		=
		\Lambda(n^{-\lambda})
		&=&
		\Big(\sum_{\substack{j | n  \\ j \neq 1, n}} c_j c_{\frac{n}{j}}\Big)^{-1}
		\sum_{\substack{d | n \\ d \neq 1, n}}
		c_{d} c_{\frac{n}{d}}
		\Lambda(d^{-\lambda})
		\Lambda\left(
		\left(
		\frac{n}{d}
		\right)^{-\lambda}
		\right)\\
		&=&
		\frac{
			\sum\limits_{\substack{ \beta \le \alpha \\ \beta \neq \mathbf 0, \alpha}}
			c_{\mathbf q^\beta} c_{\mathbf q^{\alpha - \beta}}
			\Lambda\left(\Big(\prod\limits_{j=1}^k q_j^{\beta_j}\Big)^{-\lambda}\right)
			\Lambda\left(\Big(\prod\limits_{j=1}^k q_j^{\alpha_j - \beta_j}\Big)^{-\lambda}\right)
		}{
			\sum\limits_{\substack{ \beta \le \alpha \\ \beta \neq \mathbf 0, \alpha}}
			c_{\mathbf q^\beta} c_{\mathbf q^{\alpha - \beta}}
		}\\
		&=&
		\prod_{j=1}^k \Lambda(q_j^{-\lambda})^{\alpha_j}.
	\end{eqnarray*}
	This proves the claim \eqref{multiplicative-monomials}. Consequently, since $k$ and $\{p_{m_1}, \ldots, p_{m_k}\}$ are arbitrary, the condition \eqref{multiplicativity-monomials-product} is also established.
	
	To finish the proof, let
	$
	f = \sum_{n=1}^\infty a_n n^{-\lambda}~
	\text{and}~	g = \sum_{n=1}^\infty b_n n^{-\lambda}$
	be in $\mathcal H(k)$ such that $fg \in \mathcal H(k)$. Then, by the continuity of $\Lambda$, we have
	\begin{eqnarray*}
		\Lambda(f)\Lambda(g)
		& =&
		\Bigg(\sum_{m=1}^\infty a_m \Lambda(m^{-\lambda})\Bigg)
		\Bigg(\sum_{n=1}^\infty b_n \Lambda(n^{-\lambda})\Bigg)\\
		&=&
		\sum_{m,n=1}^\infty
		a_m b_n
		\Lambda(m^{-\lambda})
		\Lambda(n^{-\lambda})\\
		&\overset{\eqref{multiplicativity-monomials-product}}=&
		\sum_{m,n=1}^\infty
		a_m b_n
		\Lambda\big((mn)^{-\lambda}\big)\\
		&=&
		\sum_{n=1}^\infty
		\Big(\sum_{d | n} a_d b_{\frac{n}{d}}\Big)
		\Lambda(n^{-\lambda})\\
		&=&
		\Lambda\Bigg(
		\sum_{n=1}^\infty
		\Big(\sum_{d | n} a_d b_{\frac{n}{d}}\Big)
		n^{-\lambda}
		\Bigg)=
		\Lambda(fg).
	\end{eqnarray*}
	
	This completes the proof.\\

\cref{theorem_Main_3} covers two broad families of spaces of Dirichlet series, given below.

\begin{itemize}
\item [$(i)$] It is well known that, for a normalized, irreducible complete Pick space, the kernel functions belong to the associated multiplier algebra; see, Proposition~2.1 of \cite{CHMR}. Consequently, \cref{theorem_Main_3} applies to Hilbert spaces of Dirichlet series associated with normalized, irreducible complete Pick kernels. More precisely, by \cite[Theorem~26]{McS} (see also \cite[Theorem~3.2]{ADS}), all such kernels are given by
$$
k(\lambda, u) = \frac{1}{1 - \sum_{j=1}^d b_j n_j^{-\lambda-\overline{u}}}, \quad \lambda, u \in \mathbb H_\rho
$$
for some $d \in \mathbb N \cup \{\infty\}$, $(b_j)_{j=1}^d \in (0, \infty)^d,$ and a $d$-tuple $(n_j)_{j=1}^d$ of positive integers greater than $1$ such that $\{\log n_i, \log n_j\}$ is linearly independent over the field of rational numbers for some distinct $i,j$ (the rational independence is needed for irreducibility).  Thus, \cref{theorem_Main_3} applies to $\mathcal H(k)$ associated with such a kernel $k$.
\item [$(ii)$] 	\cref{theorem_Main_3} also applies to a Bergman-type family of Hilbert spaces of
Dirichlet series, including the Hardy space of Dirichlet
series itself. This family was studied in \cite{Mccarthy} and is defined as
follows. Let $\mu$ be a probability measure on $[0,\infty)$ with
$0 \in \operatorname{supp}(\mu)$, and define $\{c_n\}_{n=1}^\infty$ by
$$
\frac{1}{c_n} = \int_0^\infty n^{-2t}\,d\mu(t), \qquad n \ge 1
$$
(the choice $\mu = \delta_0$ recovers the Hardy space). By
\cite[Theorem~1.11]{Mccarthy}, $ M(\mathcal H(k)) =
H^\infty(\mathbb H_0) \cap \mathcal D$ (isometrically). Here, $\mathcal D$ denotes the set of all functions which are representable by a convergent Dirichlet series in some right half-plane. The hypothesis of \cref{theorem_Main_3} is
then supplied by the following lemma, applied with $r = 0$.

\end{itemize}

\begin{lemma}\label{lemma:dirichletseries}
	Let $k(\lambda,u) = \sum_{n=1}^{\infty} c_n\, n^{-\lambda-\overline{u}}$ be a diagonal Dirichlet series kernel on $\mathbb H_\rho$. Suppose there exists $r \in \mathbb R$ such that  
	$H^\infty(\mathbb H_r) \cap \mathcal D \subseteq M(\mathcal H(k))$
	(as a set). Then, 
$k_u^2 \in M(\mathcal H(k))$ for every $u \in \mathbb H_{2\rho -r}$.
\end{lemma}
\begin{proof}
By the assumption, any series in $H^\infty(\mathbb H_r) \cap \mathcal D$ converges on $\mathbb H_\rho$, and hence $r \leq \rho$. Now, let $u \in \mathbb H_{2\rho -r} \subseteq \mathbb H_\rho$. Then, it follows that
	$r + \operatorname{Re}(u) > 2\rho$, and hence for every
	$\lambda \in \overline{\mathbb H_r}$,
	$$
	|k_u(\lambda)|
	\le \sum_{n=1}^{\infty} c_n\, n^{-(\operatorname{Re}(\lambda) + \operatorname{Re}(u))}
	\le \sum_{n=1}^{\infty} c_n\, n^{-(r + \operatorname{Re}(u))}	< \infty .
	$$
	Thus, $k_u \in H^\infty(\mathbb H_r) \cap \mathcal D \subseteq 
	M(\mathcal H(k))$. Using the fact that $M(\mathcal H(k))$ is an algebra, we obtain 
	$$
	k_u^2 = k_u\cdot k_u \in M(\mathcal H(k)),
	$$
	as claimed.
\end{proof}

There exist diagonal Dirichlet series kernels for which the conclusion of \cref{lemma:dirichletseries} fails.

\begin{example}
	Let $\{c_n\}_{n = 1}^\infty$ be a sequence defined by
	$$
	c_n
	= \begin{cases}
		\dfrac{1}{p^{8}e^{p}}, & n=p^2 \text{ for prime $p$},\\
		\dfrac{1}{n^2}, & \text{otherwise}.
	\end{cases}
	$$
	Then, it is clear that $\sum_{n=1}^{\infty} \frac{c_n}{n^{\delta}} < \infty$ for all $\delta > -1$, and $\sum_{n=1}^{\infty} nc_n$ is divergent. Hence, $k$ is a diagonal Dirichlet series kernel on $\mathbb H_{-1}$. Since
	$$
	\sum_{d \mid p^2} c_d c_{\frac{p^2}{d}}
	=
	2c_1c_{p^2}+c_p^2
	>
	c_p^2  \quad \text{for every prime $p$},
	$$
	it follows that
	$$
	\frac{\Big(\sum\limits_{d \mid p^2} c_d c_{\frac{p^2}{d}}\Big)^2}{c_{p^2}}
	>
	\frac{c_p^4}{c_{p^2}}
	=
	e^p.
	$$
	Therefore, the series
	\[
	\sum_{n=1}^{\infty}
	\frac{\left(\sum_{d \mid n} c_d c_{\frac{n}{d}}\right)^2}{c_n}
	n^{-2t} \quad \text{diverges},
	\]
 which implies that $
	k_t^2 \notin \mathcal H(k)
	$ for every $t > -1$.
\end{example}

\section{Characterizing weighted composition operators on RKHSs}

In this section, we apply our main results to characterize weighted composition operators on reproducing kernel Hilbert spaces associated with diagonal holomorphic kernels. Composition operators form a classical class of operators on spaces of analytic functions; see Shapiro's monograph~\cite{Shapiro}, and weighted composition operators provide a natural extension of this class. Weighted composition operators on Dirichlet spaces \cite{mashreghiDirichlet} have previously been characterized in terms of linear operators that preserve nowhere-vanishing functions. Building on our characterization of multiplicative linear functionals, we extend this viewpoint to a broader class of reproducing kernel Hilbert spaces, showing that a bounded operator can be characterized as a weighted composition operator solely through its action on certain functions stemming from the reproducing kernels. 

\begin{thm}
	Let $k$ be a diagonal holomorphic kernel on $\Omega$. Assume that $k_{\bw}^2$ is in $\mathcal{H}(k)$ for every $\bw \in \Omega$, and every nonzero multiplicative linear functional on $\mathcal{H}(k)$ is given by a point evaluation at some point of $\Omega$. Let $T \in \mathcal{B}(\mathcal{H}(k))$ and suppose that
	$h := T(1)$ is nowhere zero on $\Omega$. Then $T$ is a weighted
	composition operator if and only if
	\begin{equation}\label{eq:main}
		T\bigl(k_{\bw}^2\bigr)\,T(1)
		=
		\bigl(Tk_{\bw}\bigr)^2,
		\qquad
		\bw \in \Omega.
	\end{equation}
\end{thm}
\begin{proof} In this proof, we shall use $\circ$ to denote composition of functions. 
Suppose that $Tf = h (f \circ \varphi)$ for some map
	$\varphi:\Omega\to\Omega$. Then, for every $\bw \in\Omega$,
	\[
	T\bigl(k_{\bw}^{2}\bigr)=h \bigl(k_{\bw}^{2}\circ\varphi\bigr),
	\qquad
	Tk_{\bw}=h (k_{\bw}\circ\varphi),
	\qquad
	T(1)=h,
	\]
	pointwise on $\Omega$, and therefore
	\[
	T\bigl(k_{\bw}^{2}\bigr)\,T(1)
	=h^{2}\bigl(k_{\bw}^{2}\circ\varphi\bigr)
	=h^{2}(k_{\bw}\circ\varphi)^{2}
	=\bigl(h (k_{\bw}\circ\varphi)\bigr)^{2}
	=\bigl(Tk_{\bw}\bigr)^{2},
	\]
	which is \eqref{eq:main}.

 In the converse direction, assuming that \eqref{eq:main} holds, and fixing $\bw \in\Omega$, if we set $
	g_{\bw}:=T^{*}k_{\bw}\in\mathcal{H}(k)$, then for every $f\in\mathcal{H}(k)$, the reproducing property gives
	\begin{equation}\label{eq:uw}
		\langle f,g_{\bw}\rangle=\langle f,T^{*}k_{\bw}\rangle=\langle Tf,k_{\bw}\rangle
		=(Tf)(\bw).
	\end{equation}
	Taking $f=k_{\bv}$, $f=k_{\bv}^{2}$ and $f=1$ in \eqref{eq:uw} respectively yield
	\[
	\begin{gathered}
		\langle k_{\bv}, g_{\bw}\rangle=(Tk_{\bv})(\bw),\quad
		\bigl\langle k_{\bv}^{2}, g_{\bw}\bigr\rangle=\bigl(T(k_{\bv}^{2})\bigr)(\bw),\quad
		\langle 1, g_{\bw}\rangle=\bigl(T(1)\bigr)(\bw)=h(\bw),
	\end{gathered}
	\]
	for every $\bv\in\Omega$. Evaluating \eqref{eq:main} at the point $\bw$ therefore
	reads
	\begin{equation}\label{eq:lemhyp}
	\bigl\langle k_{\bv}^{2},g_{\bw}\bigr\rangle\,\langle 1,g_{\bw}\rangle
	=\langle k_{\bv},g_{\bw}\rangle^{2}, \quad \bv\in\Omega.
	\end{equation}
	We shall verify that there exists a unique $\alpha_{\bw} \in\Omega$ such that
	$
	g_{\bw}=\overline{\langle 1,g_{\bw}\rangle}\,k_{\alpha_{\bw}}.
	$

Since $h$ is nowhere zero, we have $\langle 1,g_{\bw}\rangle=h(\bw)\neq 0$, and hence the verification above  is achieved by setting $a=\langle 1,g_{\bw}\rangle$ and $g = \overline{a}^{\,-1}g_{\bw}$. \cref{eq:lemhyp} now takes the form 
	\[
	\bigl\langle k_{\bv}^{2}, g\bigr\rangle=\langle k_{\bv}, g\rangle^{2},
	\qquad \bv\in\Omega.
	\]
	Hence, $\Lambda(f):=\langle f, g\rangle$ defines a bounded linear functional
	on $\mathcal{H}(k)$ with $\Lambda(1)=1$ which satisfies $\Lambda(k_{\bv}^2)=\Lambda(k_{\bv})^2$ for every $\bv \in \Omega$. Therefore, by \cref{theorem_Main_1}, together with the assumption, there is $\alpha_{\bw}\in\Omega$
	with $\Lambda(f)=f(\alpha_{\bw})=\langle f,k_{\alpha_{\bw}}\rangle$ for all $f\in\mathcal{H}(k)$, whence $g=k_{\alpha_{\bw}}$ and
	$g_{\bw}=\overline{\langle 1,g_{\bw}\rangle} \,k_{\alpha_{\bw}}$. If also $g_{\bw}=\overline{\langle 1,g_{\bw}\rangle} \,k_{\mathbf y}$ for some
	$\mathbf{y} \in\Omega$, then $k_{\alpha_{\bw}}=k_{\mathbf{y}}$, so that
	$f(\alpha_{\bw})=\langle f,k_{\alpha_{\bw}}\rangle=\langle f,k_{\mf{y}}\rangle=f(\mf{y})$ for
	every $f\in\mathcal{H}(k)$. Since $\mathcal{H}(k)$ separates the points of $\Omega$, $\alpha_{\bw}=\mf y$, we obtain the uniqueness of $\alpha_{\bw}$.
	Therefore,  
	\[
	T^{*}k_{\bw}=g_{\bw}=\overline{\langle 1,g_{\bw}\rangle}\,k_{\alpha_{\bw}}
	=\overline{h(\bw)}\,k_{\alpha_{\bw}},
	\]
	and hence, by defining $\varphi(\bw):=\alpha_{\bw}$ for $\bw\in\Omega$, we obtain for all $f\in\mathcal{H}(k)$ and $\bw\in\Omega$ that
	\begin{equation} \label{eq:wco}
		(Tf)(\bw)=\langle f,T^{*}k_{\bw}\rangle
	=\bigl\langle f,\overline{h(\bw)}\,k_{\varphi(\bw)}\bigr\rangle
	=h(\bw)\,f\bigl(\varphi(\bw)\bigr).
	\end{equation}
 It remains to prove that $\varphi$ is holomorphic. Since $h=T(1)$ lies in
	$\mathcal{H}(k)$ and is nowhere zero on $\Omega$, the identity in \cref{eq:wco} may
	be rewritten as
	\[
	f\circ\varphi=\frac{Tf}{h}, \qquad f\in\mathcal{H}(k),
	\]
	which exhibits $f\circ\varphi$ as a holomorphic function on $\Omega$ for every
	$f\in\mathcal{H}(k)$. Using the fact that  the coordinate functions $z_j$ are in $\mathcal{H}(k)$, it follows that $z_j \circ \varphi$ is holomorphic on $\Omega$, and hence $\varphi$ is holomorphic on $\Omega$.
	
	That the symbol is unique is easy to see: if $\varphi_{1},\varphi_{2}:\Omega\to\Omega$
	satisfy $f(\varphi_{1}(\bw))=f(\varphi_{2}(\bw))$ for every $f\in\mathcal{H}(k)$ and
	$\bw\in\Omega$, then by the separation property of $\mathcal{H}(k)$, $\varphi_{1}=\varphi_{2}$.

The above proves the first part of \cref{theorem_application_main}. The other parts have similar proofs by using  \cref{completePick}. 
\end{proof}

\section{Sharpness of the Main Results and Point Evaluations}

\subsection{Examples:} \label{examples:MainTheorems}

In the two examples below, we use the fact that every multiplicative linear functional on a complete Pick space  is necessarily bounded; see, \cite{CHMR}.
The first example shows that the boundedness assumption in \cref{theorem_Main_1} and \cref{theorem_Main_3} is essential.

\begin{example}
	Let $k$ be a complete Pick kernel on $\Omega$ and $S:=\{k_{w}, k_{w}^2: w\in \Omega\}$.
Assume that the dimension of the quotient space $\mathcal H(k) /\operatorname{span}(S)$ is infinite. Let $S'$ be a maximal linearly independent subset of $S$, and let $\mathcal B$ be a Hamel basis of $\mathcal H(k)$ obtained by extending $S'$. Since $\mathcal B\setminus S'$ is infinite, we may choose a countably infinite subset $\{g_n: n\in \mathbb{N}\}\subset \mathcal B\setminus S'$. Next, fix $z_0\in \Omega$ and  define $\Lambda: \mathcal B \to \mathbb{C}$  by 
\begin{equation*}
	\Lambda(f) =
	\begin{cases}
		f(z_0), & \text{if } f \in S', \\
		m \|f\|, & \text{if } f = g_m,\\
		0, & \text{otherwise}.
	\end{cases}
\end{equation*}
Extending $\Lambda$ linearly to $\mathcal H(k)$, we obtain a linear functional that satisfies  $\Lambda(f)=f(z_0)$ for all  $f\in\operatorname{span}(S)$.
Consequently,
$$
\Lambda(k_{w}^2)=k_w^2(z_0)
=k_{w}(z_0)^2
=\Lambda(k_{w})^2,
\qquad w\in\Omega.
$$
Nevertheless, $\Lambda$ is not multiplicative, since it is unbounded. 

We now construct concrete kernels satisfying the above assumptions, which will show that the boundedness assumption in \cref{theorem_Main_1} \& \cref{theorem_Main_3} is essential. If $k$ is the Szeg\"{o} kernel on the unit disc, then $\mathcal H(k)=H^2(\mathbb D)$. Moreover, $S\subseteq A(\mathbb D)$, where $A(\mathbb D)$ denotes the disk algebra, and hence $\operatorname{span}(S)\subseteq A(\mathbb D).$ Since $A(\mathbb D)$ has infinite codimension in $H^2(\mathbb D)$, it follows that $\dim\left(H^2(\mathbb D)/\operatorname{span}(S)\right)=\infty.$ So, the construction above applies, and hence, the boundedness assumption in \cref{theorem_Main_1} is essential. 

In a similar spirit, let $k$ be a complete Pick Dirichlet series kernel given by
\[
k(\lambda, u) \;=\; \frac{1}{2-\zeta(\lambda+\overline{u})}
\;=\; \sum_{n=1}^{\infty} c_n\, n^{-\lambda-\overline{u}}, \quad \lambda, u \in \mathbb H_\rho,
\]
where $\rho$ is the unique real number satisfying $\zeta(2\rho) = 2$, and $c_n$ is the total number of ways $n$ can be factored into a product of positive integers bigger than $1$. 
To apply the above general construction, we verify that $\dim\left(\mathcal H(k)/\operatorname{span}(S)\right)=\infty$.
Define
\[
A(\mathbb H_\rho):=\bigl\{f\in\mathcal H(k)\;:\;f\ \text{extends continuously to}\ \overline{\mathbb H_\rho}\bigr\}.
\]
Clearly, $A(\mathbb H_\rho)$ is a vector subspace of $\mathcal H(k)$,
and if $A(\mathbb H_\rho)=\mathcal H(k)$, then for any
sequence $\{\sigma_m\}_{m=1}^\infty \subseteq (\rho, \infty)$ such that $\sigma_m\to\rho$, the map
\[
\Lambda_m:\ f\longmapsto \langle f,k_{\sigma_m}\rangle
\]
is a bounded linear functional on $\mathcal H(k)$, and for each $f \in \mathcal H(k)=A(\mathbb H_\rho)$,
\[
\sup_{m\ge1}\bigl|\Lambda_m(f)\bigr|
=\sup_{m\ge1}\bigl|\langle f,k_{\sigma_m}\rangle\bigr|
=\sup_{m\ge1}\bigl|f(\sigma_m)\bigr|<\infty .
\]
By the principle of uniform boundedness,
$
\sup_{m \geq 1}\|\Lambda_m\|<\infty.
$
On the other hand, since $k$ blows up at $\rho$ and $k(\sigma_m,\sigma_m)$ increases as $\sigma_m\to\rho$, we get
\[
\sup_{m\geq 1}\|\Lambda_m\|^{2}
=\sup_{m\geq 1}\|k_{\sigma_m}\|^{2}
=\lim_{m \to \infty}k(\sigma_m,\sigma_m)
=\infty.
\]
Therefore, $A(\mathbb H_\rho)\subsetneq \mathcal H(k)$. Now, let $f\in\mathcal H(k)$ be such that $f \notin A(\mathbb H_\rho)$ and define 
\[
\varphi_n:=f+n^{-\lambda},\qquad n\ge1.
\]
Since $f \notin A(\mathbb H_\rho)$, we obtain that $\varphi_n\notin A(\mathbb H_\rho)$ for all $n\ge1$. For the linear independence of $\{\varphi_j\}_{j=1}^\infty$, assume for a finite subset $J$ that  $\sum_{j \in J}d_j\varphi_j=0$. This implies that 
\[
\Bigl(\sum_{j \in J}d_j\Bigr)f+\sum_{j \in J}d_j\,j^{-\lambda}=0.
\]
Since $f \notin A(\mathbb H_\rho)$, the identity above forces that  $\sum_{j \in J}d_j=0$, and hence 
$\sum_{j \in J}d_j\,j^{-\lambda}=0.$ Therefore, by the uniqueness of Dirichlet series, $d_j=0$  for all $j \in J.$ Hence, $A(\mathbb H_\rho)$ has infinite codimension in $\mathcal H(k)$, which implies that $\dim\left(\mathcal H(k)/\operatorname{span}(S)\right)=\infty.$ So, the construction above applies and hence, the boundedness assumption in \cref{theorem_Main_3} is essential.
	\end{example}

The next example shows that boundedness is essential in  \cref{completePick}.

\begin{example}
	Consider the subset 
	$$
	S:=\{k_w : w\in \mathbb{D}\}\cup \{z\}
	$$
	of the Hardy space $H^2(\mathbb{D})$. We claim that $S$ is linearly independent. Indeed, suppose that for some $w_1,\dots,w_n\in\mathbb{D}$ and scalars $c_1,\dots,c_n,c$, we have
	$$
	c_1 k_{w_1}+\cdots+c_n k_{w_n}+c z=0.
	$$
	Taking the inner product with an arbitrary $f\in H^2(\mathbb{D})$, and using the reproducing properties
	$
	\langle f,k_w\rangle = f(w), \text { and } 
	\langle f,z\rangle = f'(0),
	$
	we obtain
	$$
	\sum_{j=1}^n \overline{c_j}\, f(w_j) + \overline{c}\, f'(0)=0
	$$
	for every $f\in H^2(\mathbb{D})$. Now choose a function $f$ such that $f(w_j)=0$ for all $j$ and $f'(0)\neq 0$. Substituting this function into the above identity gives $\overline{c}\, f'(0)=0$, and hence $c=0$. The identity then reduces to
	$$
	\sum_{j=1}^n \overline{c_j}\, f(w_j)=0
	$$
	for all $f\in H^2(\mathbb{D})$. Since the set $\{k_{w_j}\}_{j=1}^n$ is linearly independent, it follows that $c_j =0$ for all $j=1, \ldots, n$. Therefore, $S$ is linearly independent.
	
	We may extend $S$ to a Hamel basis $\mathcal{B}$ of $H^2(\mathbb{D})$. Since $S \subseteq A(\mathbb D)$, it follows that $\mathcal{B}\setminus S$ must be infinite. Therefore, we may choose a countably infinite subset
	$$
	\{g_n : n\in \mathbb{N}\} \subset \mathcal{B} \setminus S.
	$$
	Fix $z_0\in \mathbb{D} $ and  define $\Lambda: \mathcal{B} \to \mathbb{C}$ given by 
	\begin{equation*}
		\Lambda(f) =
		\begin{cases}
			f(z_0), & \text{if } f \in S, \\
			m \|f\|, & \text{if } f = g_m,\\
			0, & \text{otherwise}.
		\end{cases}
	\end{equation*}
	We extend $\Lambda$ linearly to $H^2(\mathbb{D})$. Since $\Lambda(k_w)=k_w(z_0)$ for all $w\in \mathbb D$, so  
	\begin{equation*}
		\Lambda\left(\frac{1}{k_w}\right)=\Lambda(1-\overline{w}z)=1-\overline{w}z_0=\frac{1}{k_w(z_0)}=\frac{1}{\Lambda(k_w)}.
	\end{equation*}
	Thus, it satisfies the required condition, but it is not bounded. Hence, $\Lambda$ cannot be multiplicative on $H^2(\mathbb{D})$.
\end{example} 

\subsection{Multiplicative linear functional as a point evaluation} \label{Mult_but_not_point}

There are many known examples of complete Pick spaces in which multiplicative linear functionals are not given by point evaluations; see, for instance, the discussion in Section~6 of \cite{CHMR}, where such an example is constructed on the open unit disc—namely, the local Dirichlet space at a point on the circle.  One can verify that these examples also satisfy the hypotheses of  \cref{completePick}.  

Proposition 9 and the discussion surrounding it in \cite{McS} tell us that, in the situations of \cref{theorem_Main_1} or \cref{completePick} of this paper, a multiplicative linear functional need not be a point evaluation at a point in the domain on which the functional Hilbert space is defined. It is, however, a point evaluation at a point in a larger subset of $\mathbb{C}^d$, the so-called maximal domain; see \cite{CM}. Obtaining the maximal domain in specific cases involves specific methods, see \cite{CHMR}, page 1129 and the references therein. The situation of \cref{theorem_Main_3} is quite different. Here a multiplicative linear functional, which is not an evaluation at a point in the domain, always exists as the following lemma shows. Moreover, the maximal domain is not contained in $\mathbb C$ where the domain is. In the case when the $c_n$ in the diagonal Dirichlet kernel $k(\lambda , u) = \sum_{n=1}^\infty c_n n^{-\lambda - \overline{u}}$ are completely multiplicative, i.e., $c_{mn} = c_m c_n$ for $m,n \ge 1$, the maximal domain is 
 \begin{equation}\label{max}
 \{ (x_j)_{j=1}^\infty : (\sqrt{c_{p_j}} x_j)_{j=1}^\infty \in \mathbb D^\infty \cap \ell^2(\mathbb N)\}.
 \end{equation}
 Here $p_j$ is an enumeration of the primes as usual. Any multiplicative linear functional on $\mathcal H(k)$ for $k$ of the above form is a point evaluation for a point in  the maximal domain \eqref{max}. The calculations are left to the reader. 
 
\begin{lemma} \label{Always_mlf_not_pointev}
Let $k$ be a diagonal Dirichlet series kernel on $\mathbb H_\rho$. Then, there exist multiplicative linear functionals on $\mathcal H(k)$ which are not given by point evaluations. 
	\end{lemma}
\begin{proof}
	Let $k(\lambda , u) = \sum_{n=1}^\infty c_n n^{-\lambda - \overline{u}},~ \lambda, u  \in \mathbb H_\rho$, with $c_n > 0$ for all $n \geq 1$.  Then, there exists $D>0$ such that 
	\begin{equation} \label{seq-bdd}
		c_n \leq D n^{2(\rho +1)}, \quad n \geq 1.
	\end{equation}
	Let $t = \{t_n\}_{n=1}^\infty$ be a completely multiplicative sequence such that $t_n \geq n$ for all $n \geq 1$, and define $b_n: = \frac{c_n}{t_n^{\rho +2}},~ n \geq 1$.  Then, by \eqref{seq-bdd},
	$$
	\sum_{n=1}^\infty \frac{|b_n|^2}{c_n} = \sum_{n=1}^\infty \frac{c_n}{t_n^{2(\rho +2)}} \leq D \sum_{n=1}^\infty \frac{n^{2(\rho + 1)}}{t_n^{2(\rho + 2)}}  \leq  D \sum_{n=1}^\infty \frac{1}{n^2}  < \infty,
	$$
	giving that $h(\lambda) = \sum_{n=1}^\infty b_n n^{-\lambda} \in \mathcal H(k)$. Thus, the functional $\Lambda_t(f): = \langle f, h \rangle$ is bounded on $\mathcal H(k)$. Also, since $\{\frac{b_n}{c_n}\}_{n=1}^\infty$ is completely multiplicative,  $\Lambda_t((mn)^{-\lambda}) = \Lambda(m^{-\lambda}) \Lambda(n^{-\lambda})$ for all $m, n \geq 1$, and hence, by the continuity argument, $\Lambda_t$ is multiplicative.
	
	The functional $\Lambda_t$ is given by a point evaluation if and only if there exists $u \in \mathbb H_{\rho}$ such that $h = k_{u}$; equivalently, by the uniqueness of Dirichlet series,	$ 	b_n = c_n n^{-\overline{u}}, ~ n \geq 1.	$
	Therefore, $\Lambda_t$ is given by a point evaluation if and only if
	\begin{equation*} \label{equation:all-point-ev}
	t_n^{\rho + 2} = n^{\text{Re(u)}},  \quad n \geq 1.
	\end{equation*}
	Clearly, there are uncountably infinitely many choices of the sequence $t_n$ for which no such $u$ will exist. 
\end{proof}

{\footnotesize \noindent \textsf{Acknowledgement}: The authors are thankful to \href{https://sites.google.com/view/abhayjindal/home}{Dr. Abhay Jindal} for helpful discussions.  The first author thanks \href{https://www.fmf.uni-lj.si/en/}{Faculty of Mathematics and Physics, University of Ljubljana} and the \href{https://www.math.iitb.ac.in/}{Department of Mathematics, Indian Institute of Technology, Bombay} for visits which helped this work.  AI has been used to formulate the result on weighted composition operators and to improve writing.}

\end{document}